\documentclass[11pt]{amsart}
\usepackage{amssymb, adjustbox, enumerate, amsbsy, stmaryrd}
\usepackage{geometry}
\usepackage[T1]{fontenc}
\usepackage{amsfonts, amssymb, amscd}
\numberwithin{equation}{section}

\usepackage[symbol]{footmisc}

\usepackage{bm}
\usepackage{verbatim}
\usepackage{mathrsfs}
\usepackage{graphicx}
\usepackage{tikz-cd}
\usepackage{subcaption}
\usepackage{listings}
\usepackage{subfiles}
\usepackage[toc,page]{appendix}
\usepackage{mathtools}
\usepackage{comment}
\usepackage{enumerate}
\usepackage{enumitem}
\usepackage[all]{xy}

\usepackage{booktabs,array,enumitem}
\graphicspath{{images/}}

\usepackage{appendix}
\usepackage{hyperref}
\hypersetup{
    colorlinks=true,
    citecolor=red,
    linkcolor=blue,
    filecolor=magenta,      
    urlcolor=red,
}
\newcommand{\cone}[1]{\langle #1\rangle}

\newcommand{\Spec}{\mathrm{Spec}}

\newcommand{\Cc}{\mathbb{C}}

\newcommand{\Pp}{\mathbb{P}}
\newcommand{\Qq}{\mathbb{Q}}

\newcommand{\Rr}{\mathbb{R}}

\newcommand{\Zz}{\mathbb{Z}}

\newcommand{\Exc}{\operatorname{Exc}}

\newcommand{\Hom}{\operatorname{Hom}}

\newcommand{\mult}{\operatorname{mult}}

\newcommand{\Oo}{\mathcal{O}}

\newtheorem{thm}{Theorem}[section]

\newtheorem{lem}[thm]{Lemma}
\newtheorem{prop}[thm]{Proposition}

\theoremstyle{definition}
\newtheorem{defn}[thm]{Definition}

\theoremstyle{definition}

\newtheorem{ex}[thm]{Example}

\theoremstyle{definition}

\begin{document}

\tikzset{fanedge/.style={line width=.35pt},fandot/.style={circle,fill=black,inner sep=.65pt},
 fanlabel/.style={font=\scriptsize,fill=white,inner sep=.6pt},
 paneltitle/.style={font=\scriptsize,align=center}}

\title[Birational morphism between threefolds]{Decompositions of birational toric morphism between threefolds}
\author{Lingyao Xie}

\address{School of Mathematical Sciences, Zhejiang University, Hangzhou, China}
\email{xielingyao@zju.edu.cn}

\begin{abstract}
Let $f:X\to Y$ be a birational morphism between smooth projective threefolds, Koll\'ar asked that if $f$ could be decomposed into smooth blow-downs and smooth flops. In this paper we show that the toric version is true.
\end{abstract}


\maketitle
\tableofcontents

\section{Introduction}
    
We work over $\mathbb{C}$. Let $f: X\to Y$ be a birational morphism between two smooth surfaces, then the Castelnuovo's contraction theorem tells us that $f$ can be decomposed into $X=X_0\to X_1\to\cdots\to X_n=Y$, where each $X_i\to X_{i+1}$ is a smooth blow-down of a $(-1)$-curve $E_i\simeq\Pp^1$. As a threefold analogue, Koll\'ar has asked in \cite[Question 4]{Kol08}:


\vspace{.5em}
{\it Let $f: X\to Y$ be a birational morphism between smooth projective threefolds, can $f$ be writen as a composite of smooth blow-downs and flops?}
\vspace{.1em}

\noindent where he also indicates the necessity of flops: Let $C\subset\Pp^3$ be an irreducible curve with a unique singular point $q$ which is either a node or a cusp, and $p\in {\rm Bl}_{C}\Pp^3$ be the (unique) singular point. Then ${\rm Bl}_p{\rm Bl}_C\Pp^3$ is smooth but unlike ${\rm Bl}_{C'}{\rm Bl}_q\Pp^3$ (where $C'$ is the birational transform of $C$),
\[
{\rm Bl}_p{\rm Bl}_C\Pp^3\to \Pp^3
\]
cannot be written as a composite of smooth blow-ups. In some sense, one needs a new elementary operation to describe the small birational
\[
{\rm Bl}_p{\rm Bl}_C\Pp^3\dasharrow {\rm Bl}_{C'}{\rm Bl}_q\Pp^3
\]
In general the concept of flops could be subtle, but for a (smooth) threefold ($K_X$) flop $f: X\dasharrow X'$, if we let $f: X\to Z$ and $f': X'\to Z$ be the small flopping contractions, we have
\begin{itemize}
    \item $X'$ is also smooth, and
    \item $\Exc(f)$ and $\Exc(f')$ are unions of smooth rational curves.
\end{itemize}
By a local decomposition argument, one may further require the flop to be simple in the sense that both $f$ and $f'$ only contract a single $\Pp^1$ (with the sacrifice that $X'$ might not be projective). For example, the above ${\rm Bl}_p{\rm Bl}_C\Pp^3\dasharrow {\rm Bl}_{C'}{\rm Bl}_q\Pp^3$ is a simple (Atiyah) flop, the normal bundle of the contracted $\Pp^1$ is $\Oo_{\Pp^1}(-1)^{\oplus 2}$.

The question is whether the flop operation, along with smooth blow-downs, would be enough to generalize Castelnuovo's theorem to threefolds. When $\rho(X/Y)=1$, this is almost trivial, and Chen Jiang settled the $\rho(X/Y)=2$ case in \cite{Jia25}. In general the question is still highly open and all the existing examples seem to support this statement. 

A natural effort is to test the question in the toric category: the blow-down and flop operations are extremely easy to describe using the toric language (see \ref{subsec: toric blow-down and flop}) and one might expect to prove a purely combinatorial argument. This still turns out to be difficult and complicated, actually Example \ref{ex: counter-ex when not proj} shows that the clean combinatorial statement is false in general for proper toric varieties/morphisms. However, we manage to prove the toric version with an MMP guideline from the algebraic geometry, whose existence is guaranteed by the projectiveness. 

\begin{thm}\label{thm: decomp bir morph btw toric 3-folds}
Let $f: X\to Y$ be a projective, birational, and toric morphism between (quasi) projective toric (smooth) threefolds, then we can write $f$ as
\[
X=X_0\dasharrow X_1\dasharrow\cdots\dasharrow X_n=Y
\]
such that $X_i\dasharrow X_{i+1}$ is one of the following (see \ref{subsec: toric blow-down and flop}):
\begin{enumerate}
    \item a (toric) $K_{X_i}$-flop, or
    \item a smooth (toric) point blow-down, or
    \item a smooth (toric) curve blow-down.
\end{enumerate}
Furthermore, we may ask $X_i\dasharrow Y$ to be projective morphisms.
\end{thm}

\noindent\textit{Strategy}: The basic idea is simple, we can run a (toric) $K_X$-MMP over $Y$
\[
X=X_0\dasharrow X_1\dasharrow\cdots\dasharrow X_n=Y
\]
which terminates with $Y$ itself and we may assume $X_i=X_{\Sigma_i}$ is defined by the fan $\Sigma_i$. The issues are 
\begin{enumerate}
    \item $X_i\dasharrow X_{i+1}$ are flips and divisorial contractions, not flops and smooth blow-downs.
    \item $\Sigma_i$ may not be regular, i.e $X_i$ is not smooth.
\end{enumerate}
The good news is those $\Sigma_i$ are simplicial and terminal, i.e $X_{\Sigma_i}$ have terminal $\Qq$-factorial singularities, and they are well classified. For a simplicial terminal fan $\Sigma_i$, there is a canonical regular fan $\Sigma^c_i$ inducing a canonical (minimal) toric resolution $X_{\Sigma_i^c}\to X_{\Sigma_i}$. The key point is to show that the canonical lifting $$X_{\Sigma_i^c}\dasharrow X_{\Sigma_{i+1}^c}$$ of each step of the MMP $X_i\dasharrow X_{i+1}$ could be decomposed into flops and smooth blow-downs. Then we are done since $\Sigma_0^c=\Sigma_0,~ \Sigma_n^c=\Sigma_n$.

\section{Preliminary}
We recommend the readers to \cite{CLS10} for a detailed introduction of toric varieties. We will recall some basic facts here without rigorous proofs.

In our paper we only deal with threefolds, let $N:=\Zz^3$ and $M:=\Hom(N,\Zz)=N^*$.
A fan $\Sigma$ in $N_\Rr\simeq\Rr^3$ denotes a collection of strongly convex, rational polyhydral cones $\sigma\subseteq N_\Rr$ such that 
\begin{enumerate}
    \item $\forall \sigma\in\Sigma$, every face of $\sigma$ is also in $\Sigma$.
    \item $\forall \sigma,\sigma'\in\Sigma$, $\sigma\cap\sigma'$ is a face of both $\sigma$ and $\sigma'$.
\end{enumerate}
The dual cone $\sigma^{\vee}\subseteq M_\Rr$ is $\{w\in M~|~(w,v)\ge0,\forall v\in\sigma\}$. Then we can define an affine toric variety
\[
X_\sigma:=\Spec (\Cc[\sigma^\vee\cap M]).
\]
Notice that $\Cc[\sigma^\vee\cap M]\to \Cc[\tau^\vee\cap M]$ is a localization that induces an open embedding $X_\tau\to X_\sigma$ if $\tau$ is a face of $\sigma$, we can glue those $X_\sigma$ together via their common faces and get the toric variety $X_\Sigma$. It worth mentioning that the zero cone is a face of every $\sigma$ and $X_{\cone{0}}=\Spec(\Cc[M])\simeq (\Cc^*)^3$ is the common tori.

If $\dim \sigma=0$, then $\sigma=\cone{0}$.

If $\dim \sigma=1$, then $\sigma$ is a ray generated by a unique primitive vector $v\in N$ (we also use $A,B,C,...$) and sometimes we just write $A$ for (the cone) $\cone{A}$. They are in one-to-one correspondence to the toric invariant prime divisors in $X_\Sigma$. 

If $\dim \sigma=2$, then $\sigma$ can be written as $\cone{A,B}$ (usually denoted by $AB$), where $A,B$ are the primitive generators of $\sigma$'s two 1-dimensional faces. They are in one-to-one correspondence to the irreducible toric invariant curves in $X_\Sigma$, which are all smooth rational. 

If $\dim \sigma=3$, then $\sigma$ is maximal. They are in one-to-one correspondence to the toric invariant points in $X_\Sigma$.

We say $\Sigma$ is {\it simplicial} if any $\sigma\in\Sigma$ is simplicial, i.e. the cone $\sigma$ can be generated by a set of linear independent vectors. If $\dim\sigma=0,1,2$, it is automatically simplicial. If $\dim\sigma=3$, it implies that $\sigma=\cone{A,B,C}$, where $A,B,C$ are primitive and we often write it as $ABC$. One can show that $\Sigma$ is simplicial if and only if $X_\Sigma$ is $\Qq$-factorial.

We say $\Sigma$ is {\it regular} if any $\sigma\in\Sigma$ is regular, i.e. $\sigma$ is simplicial and the set of linear independent vectors above can be extended to a $\Zz$ basis of $N$. Otherwise we call it \emph{singular}. 
We define the multiplicity of a simplicial maximal cone $\sigma=\cone{A,B,C}$ as 
\[
\mult(\sigma):=|\det(A,B,C)|,
\]
which is also the lattice \emph{index} of $\Zz A+\Zz B+\Zz C\subseteq N$. A maximal cone $\sigma$ is regular iff $\mult\sigma=1$. One can show $X_\Sigma$ is smooth if and only if $\Sigma$ is regular.

If $\Sigma'$ is a subdivison of $\Sigma$ in $N$, i.e. any $\sigma'\in \Sigma'$ is contained in some $\sigma\in\Sigma$ and $|\Sigma'|=|\Sigma|$, where the support $|\Sigma|:=\cup_{\sigma\in\Sigma}\sigma$, then it induces a toric proper birational morphism $f:X_{\Sigma'}\to X_{\Sigma}$ using the affine description above. $f$ is small iff the set of 1-dimensions cones $\Sigma(1)\subseteq\Sigma$ equals $\Sigma'(1)$. In this paper we often (locally) illustrate a fan by looking at its intersection with certain affine hyperplanes in $N$ so that 1-dimensional cones becomes points $A,B,C,...$.

\subsection{Blowdowns and flops}\label{subsec: toric blow-down and flop}
Let $ABC$ be a maximal simplicial cone in a fan $\Sigma$, and $W\in N$ is a primitive vector in the interior of $ABC$. We can replace $ABC$ by 
\[
ABW \ , \ ACW \ , \ BCW
\]
and get a new fan $\Sigma'$, which is the (star) subdivision of $\Sigma$ by inserting $W$, inducing a blow-up morphism $f: X_{\Sigma'}\to X_{\Sigma}$ at the invariant point $x$ determined by $ABC$. There is an invariant divisor determined by $W$.

If $ABC$ is regular so $X_\Sigma$ is smooth at $x$, we can write $W=aA+bB+cC$ for $(a,b,c)\in\Zz^3_{>0}$, then $f$ is actually the weighted blow-up at $x$ with weight $(a,b,c)$ and the exceptional divisor is isomorphic to $\Pp(a,b,c)$. If $a=b=c=1$, then $f$ is the ordinary point blowup. We also call $f$ a smooth point blow-down.
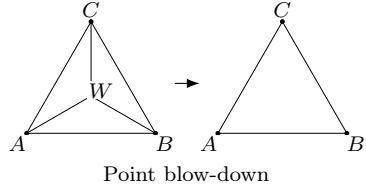
\begin{figure}[htbp]
\centering
\begin{tikzpicture}[scale=1.1, >=Latex]

\begin{scope}[shift={(9.100,0.000)}]
\coordinate (p0) at (0.77500,1.34234);
\coordinate (p1) at (1.55000,0.00000);
\coordinate (p2) at (0.00000,0.00000);
\coordinate (p3) at (0.77500,0.44745);
\draw[fanedge] (p0)--(p1);
\draw[fanedge] (p0)--(p2);
\draw[fanedge] (p0)--(p3);
\draw[fanedge] (p1)--(p2);
\draw[fanedge] (p1)--(p3);
\draw[fanedge] (p2)--(p3);
\node[fandot] at (p0) {};
\node[fanlabel] at (0.77500,1.48234) {$C$};
\node[fandot] at (p1) {};
\node[fanlabel] at (1.66000,-0.13000) {$B$};
\node[fandot] at (p2) {};
\node[fanlabel] at (-0.11000,-0.13000) {$A$};
\node[fandot] at (p3) {};
\node[fanlabel] at (0.89500,0.51745) {$W$};
\end{scope}
\begin{scope}[shift={(11.400,0.000)}]
\coordinate (p0) at (0.77500,1.34234);
\coordinate (p1) at (1.55000,0.00000);
\coordinate (p2) at (0.00000,0.00000);
\draw[fanedge] (p0)--(p1);
\draw[fanedge] (p0)--(p2);
\draw[fanedge] (p1)--(p2);
\node[fandot] at (p0) {};
\node[fanlabel] at (0.77500,1.48234) {$C$};
\node[fandot] at (p1) {};
\node[fanlabel] at (1.66000,-0.13000) {$B$};
\node[fandot] at (p2) {};
\node[fanlabel] at (-0.11000,-0.13000) {$A$};
\end{scope}
\draw[->] (10.88,.60)--(11.18,.60);
\node[paneltitle] at (11.02,-.49) {Point blow-down};
\end{tikzpicture}
\end{figure}

Let $AB$ be a 2-dimensional cone in $\Sigma$ and $U\in N$ a primitive vector in the interior of $AB$. We can insert $U$ and replace every adjacent cone $ABC$ by
$AUC,UBC$ and get a new fan $\Sigma'$. Thus, if $AB$ is an interior wall with adjacent cones $ABC,ABD$, then the subdivision replaces them with
\[
UAC,\ UBC\ ;\ UAD,\ UBD.
\]
If $AB$ is on the boundary of the support, then there is only one adjacent maximal
cone $ABC$ and the subdivision just replace it with $UAC,UBC$. In either case the induced $f: X_{\Sigma'}\to X_{\Sigma}$ is a blow-up morphism along the invariant curve determined by $AB$.

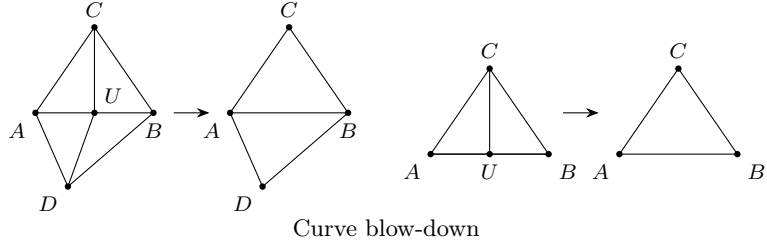
\begin{figure}[htbp]
\centering
\begin{tikzpicture}[
    x=0.78cm, y=0.78cm,
    font=\scriptsize,
    line cap=round,
    line join=round,
    >=Stealth,
    edge/.style={line width=0.4pt}
]

\begin{scope}
    \coordinate (A) at (0,0);
    \coordinate (B) at (2,0);
    \coordinate (C) at (1,1.45);
    \coordinate (D) at (0.55,-1.25);
    \coordinate (U) at (1,0);

    \draw[edge] (A)--(C)--(B)--(D)--cycle;
    \draw[edge] (A)--(U)--(B);
    \draw[edge] (U)--(C);
    \draw[edge] (U)--(D);

    \foreach \P in {A,B,C,D,U}
        \fill (\P) circle (1.2pt);

    \node[below left] at (A) {$A$};
    \node[below] at (B) {$B$};
    \node[above] at (C) {$C$};
    \node[below left] at (D) {$D$};
    \node[above right] at (U) {$U$};
\end{scope}

\draw[->] (2.35,0)--(2.95,0);


\begin{scope}[shift={(3.3,0)}]
    \coordinate (A) at (0,0);
    \coordinate (B) at (2,0);
    \coordinate (C) at (1,1.45);
    \coordinate (D) at (0.55,-1.25);

    \draw[edge] (A)--(C)--(B)--(D)--cycle;
    \draw[edge] (A)--(B);

    \foreach \P in {A,B,C,D}
        \fill (\P) circle (1.2pt);

    \node[below left] at (A) {$A$};
    \node[below] at (B) {$B$};
    \node[above] at (C) {$C$};
    \node[below left] at (D) {$D$};
\end{scope}

\node[font=\footnotesize] at (5.95,-1.95) {Curve blow-down};

\begin{scope}[shift={(6.7,0)}]
    \coordinate (A) at (0,-0.7);
    \coordinate (B) at (2,-0.7);
    \coordinate (C) at (1,0.75);
    \coordinate (U) at (1,-0.7);

    \draw[edge] (A)--(C)--(B)--cycle;
    \draw[edge] (A)--(U)--(B);
    \draw[edge] (U)--(C);

    \foreach \P in {A,B,C,U}
        \fill (\P) circle (1.2pt);

    \node[below left] at (A) {$A$};
    \node[below right] at (B) {$B$};
    \node[above] at (C) {$C$};
    \node[below] at (U) {$U$};
\end{scope}

\draw[->] (8.95,0)--(9.55,0);


\begin{scope}[shift={(9.9,0)}]
    \coordinate (A) at (0,-0.7);
    \coordinate (B) at (2,-0.7);
    \coordinate (C) at (1,0.75);

    \draw[edge] (A)--(C)--(B)--cycle;

    \foreach \P in {A,B,C}
        \fill (\P) circle (1.2pt);

    \node[below left] at (A) {$A$};
    \node[below right] at (B) {$B$};
    \node[above] at (C) {$C$};
\end{scope}

\end{tikzpicture}

\label{fig: toric curve blow-down}
\end{figure}

If every adjacent maximal cone of $AB$ is regular and $U=A+B$, then $f$ is the ordinary blow-up along a smooth curve and we also call it a smooth curve blow-down.

Suppose a fan $\Sigma$ contains two maximal cones $ABC,ABD$ sharing the same face $AB$ and
\[
aA+bB=cC+dD \text{ for some } (a,b,c,d)\in \Zz^4_{>0},
\]
so that $\Omega=\cone{A,B,C,D}$ is strictly convex. Replacing $ABC,ABD$ with $\Omega$ (resp. $ACD,BCD$) gives us a fan $\hat\Sigma$ (resp. $\Sigma'$). Then $\Sigma$ and $\Sigma'$ can be regarded as two subdivisions of $\hat\Sigma$ by adding $AB$ or $CD$, inducing two small toric proper birational morphisms
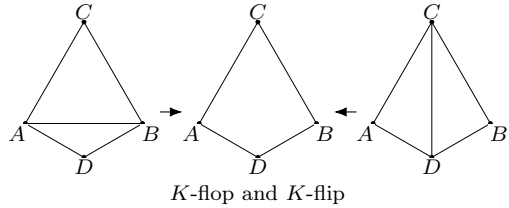
\[
f:X_{\Sigma}\to X_{\hat\Sigma}~,~f':X_{\Sigma'}\to X_{\hat\Sigma}
\]

\begin{figure}[htbp]
\centering
\begin{tikzpicture}[>=Latex]
\begin{scope}[shift={(0.000,0.000)}]
\coordinate (p0) at (0.77500,1.78979);
\coordinate (p1) at (1.55000,0.44745);
\coordinate (p2) at (0.00000,0.44745);
\coordinate (p3) at (0.77500,0.00000);
\draw[fanedge] (p0)--(p1);
\draw[fanedge] (p0)--(p2);
\draw[fanedge] (p1)--(p2);
\draw[fanedge] (p1)--(p3);
\draw[fanedge] (p2)--(p3);
\node[fandot] at (p0) {};
\node[fanlabel] at (0.77500,1.92979) {$C$};
\node[fandot] at (p1) {};
\node[fanlabel] at (1.66000,0.31745) {$B$};
\node[fandot] at (p2) {};
\node[fanlabel] at (-0.11000,0.31745) {$A$};
\node[fandot] at (p3) {};
\node[fanlabel] at (0.77500,-0.13000) {$D$};
\end{scope}
\begin{scope}[shift={(2.300,0.000)}]
\coordinate (p0) at (0.77500,1.78979);
\coordinate (p1) at (1.55000,0.44745);
\coordinate (p2) at (0.00000,0.44745);
\coordinate (p3) at (0.77500,0.00000);
\draw[fanedge] (p0)--(p1);
\draw[fanedge] (p0)--(p2);
\draw[fanedge] (p1)--(p3);
\draw[fanedge] (p2)--(p3);
\node[fandot] at (p0) {};
\node[fanlabel] at (0.77500,1.92979) {$C$};
\node[fandot] at (p1) {};
\node[fanlabel] at (1.66000,0.31745) {$B$};
\node[fandot] at (p2) {};
\node[fanlabel] at (-0.11000,0.31745) {$A$};
\node[fandot] at (p3) {};
\node[fanlabel] at (0.77500,-0.13000) {$D$};
\end{scope}
\draw[->] (1.78,.60)--(2.08,.60);
\begin{scope}[shift={(4.600,0.000)}]
\coordinate (p0) at (0.77500,1.78979);
\coordinate (p1) at (1.55000,0.44745);
\coordinate (p2) at (0.00000,0.44745);
\coordinate (p3) at (0.77500,0.00000);
\draw[fanedge] (p0)--(p1);
\draw[fanedge] (p0)--(p2);
\draw[fanedge] (p0)--(p3);
\draw[fanedge] (p1)--(p3);
\draw[fanedge] (p2)--(p3);
\node[fandot] at (p0) {};
\node[fanlabel] at (0.77500,1.92979) {$C$};
\node[fandot] at (p1) {};
\node[fanlabel] at (1.66000,0.31745) {$B$};
\node[fandot] at (p2) {};
\node[fanlabel] at (-0.11000,0.31745) {$A$};
\node[fandot] at (p3) {};
\node[fanlabel] at (0.77500,-0.13000) {$D$};
\end{scope}
\draw[->] (4.38,.60)--(4.08,.60);

\node[paneltitle] at (3.07,-.49) {$K$-flop and $K$-flip};

\end{tikzpicture}
\label{fig: flips and flops}
\end{figure}

If $a+b<c+d$ (resp. $a+b>c+d$), then $X_{\Sigma}\dasharrow X_{\Sigma'}$ (resp. $X_{\Sigma'}\dasharrow X_{\Sigma}$) is a $K$-flip, where $K$ stands for the canonical divisor. If $a+b=c+d$, then $f,f'$ are flopping contractions and $X_{\Sigma}\dasharrow X_{\Sigma'}$ is a $K$-flop.

If $ABC, ABD$ are regular, then we see $X_{\Sigma}\dasharrow X_{\Sigma'}$ is a $K$-flop if and only if
\[A+B=C+D.\]

\subsection{Canonical resolution for terminal simplicial fans}\label{subsec: terminal fans}

We say $\Sigma$ is \emph{terminal} if $X_{\Sigma}$ has terminal singularities. One can easily check that every 2-dimensional cone must be regular, which corresponds to the fact that terminal threefolds has isolated singularities.


The famous terminal lemma tells us a singular simplicial terminal 3-dimensional cone $\sigma$ has the form
\[
\sigma=\cone{e_1,e_2,e_3},\ \
 N\cong\Zz^3+\Zz\frac{(1,a,r-a)}r,
 \ \ 0<a<r,\ \ \gcd(a,r)=1. 
\]
Here $e_1=(1,0,0),e_2=(0,1,0),e_3=(0,0,1)$ generate a sublattice $\Zz^3\subseteq N$ of index $r$ so $\mult\sigma=r$. (Notice that they are not a basis
of $N$.) We also call this cone of type $\frac{1}{r}(1,a,r-a)$. Using this coordinate, the vector
\[
w=\frac{e_1+ae_2+(r-a)e_3}{r}=(\frac{1}{r},\frac{a}{r},\frac{r-a}{r})=\frac{1}{r}(1,a,r-a)\in N
\]
is the unique $v=(v_1,v_2,v_3)\in N$ in the interior of $\cone{e_1,e_2,e_3}$ such that $|v|_\sigma:=v_1+v_2+v_3=1+\frac{1}{r}$. If we think the primitive $v$ as divisorial (toric) valuations, then $|v|_\sigma$ actually computes the log discrepancy of $X_\sigma$ at $v$ and $w$ is the unique place computing the minimal log discrepancy (mld for short). Therefore $w$ is canonically determined by $\sigma$, independent of the coordinate representations.

We call $w$ the \emph{mld-ray} of $\sigma$ and we can define a star subdivision by inserting $w$, i.e. replacing $\cone{e_1,e_2,e_3}$ with 
\[
\cone{w,e_2,e_3},\cone{e_1,w,e_3},\cone{e_1,e_2,w}.
\]

\begin{lem}\label{lem: first insertion}
$\cone{w,e_2,e_3},\cone{e_1,w,e_3},\cone{e_1,e_2,w}$ are terminal and the multiplicities/indices are $1,a,r-a$ respectively.
\end{lem}
\begin{proof}
Notice that $w,e_1,e_2,e_3$ can generate $N$ and satisfy the equality
\[
rw=e_1+ae_2+(r-a)e_3,
\]
so $w,e_2,e_3$ can generate $N$ and $\cone{w,e_2,e_3}$ is regular, with multiplicity 1. For $\cone{e_1,w,e_3}$, we can write
\begin{align*}
 N=&\Zz e_1+\Zz w+\Zz e_3+\Zz\frac{1}{a}(-e_1+rw-(r-a)e_3) \\
  =&\Zz e_1+\Zz w+\Zz e_3+\Zz\frac{1}{a}(e_1-rw+re_3)\\
  =&\Zz e_1+\Zz w+\Zz e_3+\Zz\frac{1}{a}(e_1+sw+(a-s)e_3),
\end{align*}
where $0<s<a, ~s\equiv -r\mod a$, therefore by the terminal lemma $\cone{e_1,w,e_3}$ is terminal with index $a$. Similarly, $\cone{e_1,e_2,w}$ is terminal with index $r-a$.
\end{proof}

For a terminal simplicial fan $\Sigma$, we can define a number 
\[
G(\Sigma)=\sum_{\sigma\in\Sigma(3)}(\mult(\sigma)-1),
\]
where the sum is over all the 3-dimensional cones $\tau$. If we make the star subdivision for a singular $\sigma$ as described above, then the new fan $\Sigma'$ has
\[
G(\Sigma')=G(\Sigma)-(r-1)+(a-1)+(r-a-1)=G(\Sigma)-1
\]
by Lemma \ref{lem: first insertion}. Selecting singular maximal cones $\sigma$ and making the insertions repeatedly, we will end up with a regular fan $\Sigma^c$.
\begin{defn}\label{def: canonical resolution}
We define the subdivison $\Sigma^c$ (resp. $X_{\Sigma^c}$) to be the canonical resolution of $\Sigma$ (resp. $X_\Sigma$), and we have the following properties:
\begin{enumerate}
    \item The construction $\Sigma\rightsquigarrow\Sigma^c$ fixes every regular cone.
    \item $\Sigma^c$ is unique and does not depend on the order (although $\Sigma\rightsquigarrow\Sigma^c$ does).
    \item Taking the canonical resolution commutes with restriction to subfans.
    \item If $\Sigma'$ appears in the middle step of our construction $\Sigma\rightsquigarrow\Sigma^c$, then $\Sigma'^{c}=\Sigma^c$.
\end{enumerate}
Indeed, the insertions in distinct maximal cones is unique and commute because they lie in
the respective interiors, while their common faces remain
unchanged. The independence of order and the above compatibility follows from this observation. The following graphs illustrate the canonical resolution of a terminal cone of type $\frac{1}{5}(1,2,3)$.
\end{defn}
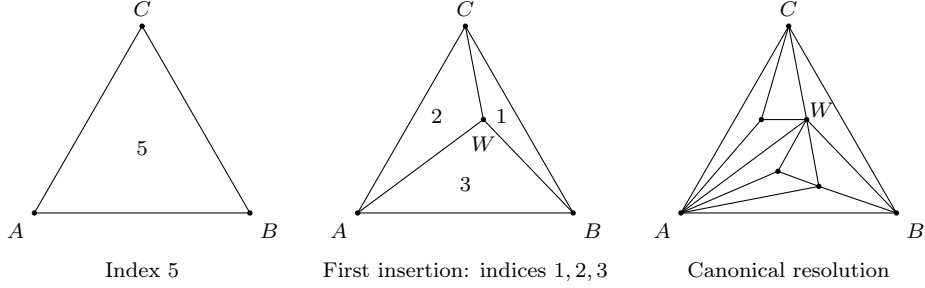
\begin{figure}[htbp]
\centering
\begin{tikzpicture}[
    x=0.95cm, y=0.95cm,
    line cap=round,
    line join=round,
    edge/.style={line width=0.4pt},
    every node/.style={font=\scriptsize}
]

\begin{scope}
    \coordinate (A) at (0,0);
    \coordinate (B) at (3,0);
    \coordinate (C) at (1.5,2.6);

    \draw[edge] (A)--(B)--(C)--cycle;

    \foreach \P in {A,B,C}
        \fill (\P) circle (1pt);

    \node[below left] at (A) {$A$};
    \node[below right] at (B) {$B$};
    \node[above] at (C) {$C$};

    \node at (1.5,0.9) {$5$};

    \node[below] at (1.5,-0.55)
        {Index $5$};
\end{scope}

\begin{scope}[shift={(4.5,0)}]
    \coordinate (A) at (0,0);
    \coordinate (B) at (3,0);
    \coordinate (C) at (1.5,2.6);
    \coordinate (W) at (1.75,1.3);

    \draw[edge] (A)--(B)--(C)--cycle;

    \draw[edge] (A)--(W);
    \draw[edge] (B)--(W);
    \draw[edge] (C)--(W);

    \foreach \P in {A,B,C,W}
        \fill (\P) circle (1pt);

    \node[below left] at (A) {$A$};
    \node[below right] at (B) {$B$};
    \node[above] at (C) {$C$};
    \node[below] at (1.75,1.2) {$W$};

    \node at (1.1,1.35) {$2$};
    \node at (2.0,1.35) {$1$};
    \node at (1.5,0.4) {$3$};

    \node[below] at (1.5,-0.55)
        {First insertion: indices $1,2,3$};
\end{scope}

\begin{scope}[shift={(9,0)}]
    \coordinate (A) at (0,0);
    \coordinate (B) at (3,0);
    \coordinate (C) at (1.5,2.6);

    \coordinate (P) at (1.12,1.3);
    \coordinate (W) at (1.75,1.3);
    \coordinate (Q) at (1.35,0.58);
    \coordinate (R) at (1.92,0.37);

    \draw[edge] (A)--(B)--(C)--cycle;

    \draw[edge] (C)--(P);
    \draw[edge] (C)--(W);
    \draw[edge] (P)--(W);

    \draw[edge] (A)--(P);
    \draw[edge] (A)--(W);
    \draw[edge] (A)--(Q);
    \draw[edge] (A)--(R);

    \draw[edge] (W)--(Q);
    \draw[edge] (W)--(R);
    \draw[edge] (Q)--(R);
    \draw[edge] (R)--(B);
    \draw[edge] (W)--(B);

    \foreach \P in {A,B,C,P,W,Q,R}
        \fill (\P) circle (1pt);

    \node[below left] at (A) {$A$};
    \node[below right] at (B) {$B$};
    \node[above] at (C) {$C$};
    \node[above right] at (1.65,1.20) {$W$};

    \node[below] at (1.5,-0.55)
        {Canonical resolution};
\end{scope}

\end{tikzpicture}

\end{figure}

\subsection{Toric MMPs}\label{subsec: toric mmp}
Let $f:X\to Y$ be a projective toric morphism betweem toric varieties and $X$ is $\Qq$-factorial, then we can run toric MMP over $Y$ and it terminates by \cite{FS04}. In our setting $f$ is birational and $X, Y$ are smooth, therefore there is a toric $K_X$-MMP over $Y$:
\[
X=X_0\dasharrow X_1\dasharrow\cdots\dasharrow X_n=Y
\]
such that
\begin{itemize}
    \item each $X_i$ is terminal and $\Qq$-factorial,
    \item each $X_i\to Y$ is projective, and
    \item $X_i\dasharrow X_{i+1}$ is either a divisorial contraction or a $K$-flip.
\end{itemize}
Assume the fan of $X_i$ is $\Sigma_i$, then $X_i\dasharrow X_{i+1}$ can be classified into the following cases, where we will use the descriptions in \ref{subsec: toric blow-down and flop}.

\vspace{.5em}

\noindent\textbf{Case 1}. $X_i\to X_{i+1}$ is a divisorial contraction to an invariant curve on $X_{i+1}$ (a curve blow-down). Let $AB$ be the regular 2-dimensional cone corresponding to the curve, then $\Sigma_{i}$ is the subdivision of $\Sigma_{i+1}$ by inserting some primitive $U$ in the interior of $AB$. Since $AU,BU$ need to be regular, we must have $U=A+B$ by looking at the sub-lattice $\Zz A+\Zz B$. By Lemma \ref{lem: terminal curve blow-down} all the adjacent maximal cone is regular, and $X_i\to X_{i+1}$ is actually a smooth curve blow-down.


\vspace{.5em}

\noindent\textbf{Case 2}. $X_i\to X_{i+1}$ is a divisorial contraction to an invariant point $x\in X_{i+1}$ (a point blow-down). Let $\sigma$ be the maximal cone correponding to $x$. Then $\Sigma_i$ is a star subdivision of $\Sigma_{i+1}$ by inserting a primitive interior ray $u\in \sigma$.
\begin{enumerate}
    \item If $\sigma$ is singular, i.e. $x$ is a singular point. Then by Lemma \ref{lem: terminal star subdiv of sing cone inserts mld ray} below $u$ must be the mld-ray of $\sigma$.
    \item If $\sigma$ is regular, i.e. $x$ is a smooth point. Then by Lemma \ref{lem: terminal point blow-down to a smooth point} the inserted ray can be written as $u=aA+bB+C$, where $\sigma=\cone{A,B,C}$, $a,b>0$ and $\gcd(a,b)=1$. $X_i\to X_{i+1}$ is therefore a (weighted) $(a,b,1)$-blowup.
\end{enumerate}

\vspace{.5em}

\noindent\textbf{Case 3}. $X_i\dasharrow X_{i+1}$ is a (terminal) $K$-flip. Then the fan $\Sigma_i\rightsquigarrow\Sigma_{i+1}$ is as described in \ref{subsec: toric blow-down and flop}, where the change only takes place in some $\Omega=\cone{A,B,C,D}$,
\[
 \{ABC,ABD\}\rightsquigarrow\{ACD,BCD\}.
\]
By the classification in \cite[Theorem 3.1]{FSTU09} we may assume $ABC$ is regular and
one of the following holds:
\begin{align*}
 \mathrm{I}(a,b):\ &D=aA+bB-rC, &r=a+b,\  a,b>0,\ &\gcd(a,b)=1;\\
 \mathrm{II}(r,a):\ &D=aA+B-rC, &0<a<r,\ &\gcd(a,r)=1. 
\end{align*}
The two families overlap when $a=r-1$ and $b=1$. The readers are encouraged to check (as an exercise) that 
$ABD,ACD,BCD$ are terminal cones of index $r,b,a$.

\begin{lem}\label{lem: terminal curve blow-down}
Let $ABC$ be a terminal maximal cone and $U=A+B$, if both $ACU$ and $BCU$ are terminal. Then $ABC,ACU,BCU$ are all regular.
\end{lem}
\begin{proof}
It suffices to show $ABC$ is regular. Assume $\mult ABC=r>1$, then it is easy to check that $\mult ACU=\mult BCU=r$ by looking at the determinants. Let $w=aA+bB+cC$ be the mld-ray of $ABC$, then we have $a+b+c=1+\frac{1}{r}$ and $ra,rb,rc\in\Zz_{>0}$.

Without lose of generality, we may assume $a\ge b$ so $w\in ACU=\tau$. Now we can write
\[
w=bU+(a-b)A+cC,
\]
hence we have the log discrepancy $|w|_{\tau}=b+(a-b)+c=1+\frac{1}{r}-b\le1$. But $w$ (as a divisorial valuation) is exceptional over $X_{\tau}$, which contradicts the terminality of $X_\tau$.
\end{proof}

\begin{lem}\label{lem: terminal star subdiv of sing cone inserts mld ray}
Suppose a star subdivision of a singular terminal cone $\sigma$ at a
primitive interior ray $u$ is terminal. Then $u$ is exactly the
mld-ray $w$. 
\end{lem}
\begin{proof}
We may assume $\sigma=\cone{e_1,e_2,e_3}$ and $w=(w_1,w_2,w_3)=\frac{1}{r}(1,a,r-a)$ as in \ref{subsec: terminal fans}. Write $u=(u_1,u_2,u_3)\in N$ and notice that $u_j>0$, let 
\[
t=\frac{w_i}{u_i}=\min_j\{\frac{w_j}{u_j}\}> 0, 
\]
then $w$ is contained in the cone $\tau$ generated by $\{e_1,e_2,e_3,u\}\backslash\{e_i\}$ and we have
\[
w=(w_1-tu_1)e_1+(w_2-tu_2)e_2+(w_3-tu_3)e_3+tu,
\]
where one coefficient $w_i-tu_i=0$ and the others $\ge 0$.
Hence we can compute the log discrepancy of $w$ over $X_\tau$,
\[
|w|_\tau=\sum_j(w_j-tu_j)+t=1+\frac{1}{r}-(u_1+u_2+u_3-1)t.
\]
If $w\neq u$, then $w$ (as a divisorial valuation) is exceptional over $X_\tau$ and we must have $|w|_\tau>1$ as $X_\tau$ is terminal, equivalent to $u_i>rw_i(u_1+u_2+u_3-1)$. We show this is impossible.

If $u_1\in\Zz$, then $u_2, u_3\in\Zz$ as well. We have $u_i\le u_1+u_2+u_3-1$ and $w_i\ge \frac{1}{r}$, which implies $u_i\le rw_i(u_1+u_2+u_3-1)$.

If $u_1\notin\Zz$, then we can write $u=\bar u+\sum_jn_je_j$, where $0<\bar u_j<1,~n_j\ge 0$. Notice that $N=\Zz^3+\Zz\cdot\frac{1}{r}(1,a,r-a)$, we also have
\[
r\bar u_2\equiv (r\bar u_1)a,\ \  \   r\bar u_3\equiv r\bar u_1(r-a), \mod r
\]
In particular, we have $\bar u_2+\bar u_3=1,\ \bar u_2\le a\bar u_1,\ \bar u_3\le (r-a)\bar u_1$. Therefore for $u_1$ we have
\[u_1=rw_1(u_1+\bar u_2+\bar u_3-1)\le rw_1(u_1+u_2+u_3-1);\]
For $u_2$ we have
\begin{align*}
u_2=&\bar u_2+n_2\le a\bar u_1+an_2\le a(\bar u_1+\bar u_2+\bar u_3-1+n_2)\\
   \le& a(\bar u_1+\bar u_2+\bar u_3-1+ n_1+n_2+n_3)=rw_2(u_1+u_2+u_3);
\end{align*}
Similarly for $u_3$ we have 
\[
u_3\le (r-a)(\bar u_1+\bar u_2+\bar u_3-1+ n_1+n_2+n_3)=rw_3(u_1+u_2+u_3).
\]
Thus $u_i\le rw_i(u_1+u_2+u_3-1)$ for any $i$.
\end{proof}

\begin{lem}\label{lem: terminal point blow-down to a smooth point}
Let $W$ be a primitive vector/ray in the interior of a regular cone $ABC$. Then
\[
ABW\ ,\ AWC\ ,\ WBC
\]
are terminal if and only if we can write $W=aA+bB+C$, where $a,b>0$ and $\gcd(a,b)=1$, possibly after switching $A,B,C$.
\end{lem}
\begin{proof}
Let $W=aA+bB+cC$, we may assume $0<c\le b\le a$. Regularity of $AW$ implies $\gcd(b,c)=1$, similarly we have $\gcd(a,b)=\gcd(a,c)=1$. If $c>1$, then $\{ABW,AWC,WBC\}$ are all singular terminal cones with index $c,b,a$. Since $A=\frac{1}{a}(-bB-cC+W)$, we have 
\[
N=\Zz B+\Zz C+\Zz W+\Zz A=\Zz B+\Zz C+\Zz W+\Zz\cdot\frac{1}{a}(bB+cC+(a-1)W).
\]
Then the terminal lemma (see \ref{subsec: terminal fans}) implies that two of $\{b,c,a-1\}$ sum to $a$, which could only be $b+c=a$ since $b\ge c>1$. Since $B=\frac{1}{b}(-aA-cC+W)$, we have
\[
N=\Zz A+\Zz C+\Zz W+\Zz\cdot\frac{1}{b}(cA+cC+(b-1)W).
\]
Again the terminal lemma implies that $2c=b$, so $c=\gcd(c,b)=1$, a contradiction.

Therefore $c=1$ and $\gcd(a,b)=1$. Conversely, if we have this assumption, the same argument above show that $ABW$ is regular, $AWC,WBC$ is terminal of type $\frac{1}{b}(a,1,b-1)$ and $\frac{1}{a}(b,1,a-1)$. (One can multiply them to the standard forms in \ref{subsec: terminal fans} as $\gcd(a,b)=1$.)
\end{proof}

\section{Smooth lifts}\label{sec: sm lifts}
In this section we assume $X_{\Sigma}\dasharrow X_{\Sigma'}$ is a step of the terminal $K$-MMP classified in \ref{subsec: toric mmp}. We will show that the induced birational map between their canonical resolutions can be decomposed into smooth blow-downs and smooth $K$-flops.

\begin{prop}\label{prop: lifting mmp}
Let $f: X_{\Sigma}\dasharrow X_{\Sigma'}$ be a step of terminal toric $K$-MMP as in \ref{subsec: toric mmp}. Then there is a sequence
\[
X_{\Sigma^c}=X_0\dasharrow X_1\dasharrow\cdots\dasharrow X_n=X_{(\Sigma')^c}
\]
such that each $X_i\dasharrow X_{i+1}$ is one of the following (see \ref{subsec: toric blow-down and flop})
\begin{itemize}
    \item a smooth curve blow-down.
    \item a smooth point blow-down.
    \item a smooth $K$-flop.
\end{itemize}
Notice that it could happen that $X_{\Sigma^c}=X_0=X_{(\Sigma')^c}$.
\end{prop}

\begin{lem}\label{lem: lift curve blow-down and pt blow-down to sing pt}
Proposition \ref{prop: lifting mmp} holds in Case 1 and Case 2 (1) of \ref{subsec: toric mmp}.
\end{lem}
\begin{proof}
In Case 1, $\Sigma$ is a star subdivision of $\Sigma'$ by inserting $U=A+B$ in a regular 2-dimensional cone $AB$, dividing every adjacent maximal cone to two cones, all of them are regular. Therefore no singular cones are changed and no extra singular cones are created. By the construction of the canonical resolution and the properties in Definition \ref{def: canonical resolution}, we see $\Sigma^c$ is exactly the same star subdivision of $(\Sigma')^c$ around $AB$. In other words, the smooth curve-blown $X_\Sigma\to X_{\Sigma'}$ lifts to $X_{\Sigma^c}\to X_{(\Sigma')^c}$.

In Case 2 (1), $\Sigma$ is a star subdivision of $\Sigma'$ by inserting the mld-ray $u$ in a singular cone $\sigma'$. This can be viewed as a first step of constructing the canonical resolution of $\Sigma'$, hence by Definition \ref{def: canonical resolution}(4) we know $\Sigma^c=(\Sigma')^c$ and $X_{\Sigma^c}=X_{(\Sigma')^c}$. 
\end{proof}

We prove this by strong induction on $r$, simultaneously for the two
families. Every recursive flip below has smaller index. All intermediate
singular fans are terminal, so that is defined on them. The compatibility
in  identifies the smooth fans at the interfaces.

\begin{lem}\label{lem: lift I flip}
Proposition \ref{prop: lifting mmp} holds in Case 3, Family (I) of \ref{subsec: toric mmp}.    
\end{lem}
\begin{proof}
We will prove this with induction on $r$. Put $V=A+B-C$. Since
\[
 rV=bA+aB+D,
\]
We see $V=\frac{1}{r}(bA+aB+D)$ is the mld-ray of $ABD$: its positive coefficients have
sum $1+1/r$, which computes the mld. Inserting it gives (see \ref{subsec: terminal fans})
\[
 ABC,ABV,ADV,BDV\ \ \text{of indices }1,1,a,b,
\]
the induced $X_{\Sigma_1}\to X_{\Sigma}$ is a first step of the canonical resolution of $\Sigma$. Therefore we have $\Sigma^c=\Sigma^c_1$.

The cones $ABC,ABV$ are regular and $A+B=C+V$, so there is a $K$-flop $X_{\Sigma_1}\dasharrow X_{\Sigma_2}$ as in \ref{subsec: toric blow-down and flop}
with
\[
AB\rightsquigarrow CV,\ \ \{ABC,ABV\}\rightsquigarrow\{ACV,BCV\}.
\] 
Since no singular cones are changed or created, Definition \ref{def: canonical resolution} implies that the smooth flop lifts to their canonical resolutions $X_{\Sigma_1^c}\dasharrow X_{\Sigma_2^c}$, so we just need to show $X_{\Sigma^c_2}\dasharrow X_{(\Sigma')^c}$ have the desired decompositions. We list so far relevant maximal cones in $\Sigma_2$ here for convenience,
\begin{equation}\label{eq: I flip Sigma_2}
 ACV,BCV,ADV,BDV.
\end{equation}

If $b>a$, working in the regular basis $(B,V,C)$, we have
\[
D=(b-a)B+aV-bC.
\]
Thus by the classification in \ref{subsec: toric blow-down and flop} and \ref{subsec: toric mmp}, we see
\[
\{BCV,BDV\}\rightsquigarrow\{CDV,BCD\},\ \ \Sigma_2\rightsquigarrow \Sigma_3
\]
is a terminal flip of type $\mathrm{I}(b-a,a)$, and the singular index of $BDV$ is $b<r$. By induction we know Proposition \ref{prop: lifting mmp} holds for $X_{\Sigma_2}\dasharrow X_{\Sigma_3}$. Therefore it suffices to show that $X_{\Sigma^c_3}=X_{(\Sigma')^c}$. 

The relevant cones in $\Sigma_3$ are
\[
 ACV,ADV,CDV,BCD,
\]
and the key observation is $V$ is mld-ray of the cone $ACD\in \Sigma'$ because
\[
 bV=(b-a)A+aC+D.
\]
Therefore $\Sigma_3$ is the subdivision of $\Sigma'$ by inserting the mld-ray $V$ in $ACD$, and by Definition \ref{def: canonical resolution} we know $X_{\Sigma^c_3}=X_{(\Sigma')^c}$. 

\begin{figure}[htbp]
\centering
\begin{tikzpicture}[x=1cm,y=1cm]

\tikzset{
  flowarrow/.style={->, dashed, line width=0.5pt},
  paneltitle/.style={font=\small}
}

\begin{scope}[shift={(0.000,0.000)}]
\coordinate (p0) at (1.32500,2.00000); 
\coordinate (p1) at (2.65000,0.00000); 
\coordinate (p2) at (0.00000,0.00000); 
\coordinate (p3) at (1.70000,-1.75000); 

\draw[fanedge] (p0)--(p1);
\draw[fanedge] (p0)--(p2);
\draw[fanedge] (p1)--(p2);
\draw[fanedge] (p1)--(p3);
\draw[fanedge] (p2)--(p3);

\node[fandot] at (p0) {};
\node[fanlabel] at (1.32500,2.15000) {$C$};
\node[fandot] at (p1) {};
\node[fanlabel] at (2.79000,-0.12000) {$B$};
\node[fandot] at (p2) {};
\node[fanlabel] at (-0.12000,-0.12000) {$A$};
\node[fandot] at (p3) {};
\node[fanlabel] at (1.56000,-1.92000) {$D$};

\node at (barycentric cs:p0=1,p1=1,p2=1) {$1$};
\node at (barycentric cs:p1=1,p2=1,p3=1) {$r$};

\end{scope}
\node[paneltitle] at (4,0) {$\Sigma$};

\begin{scope}[shift={(8.200,0.000)}]
\coordinate (p0) at (1.32500,2.00000); 
\coordinate (p1) at (2.65000,0.00000); 
\coordinate (p2) at (0.00000,0.00000); 
\coordinate (p3) at (1.70000,-1.75000); 

\draw[fanedge] (p0)--(p1);
\draw[fanedge] (p0)--(p2);
\draw[fanedge] (p0)--(p3);
\draw[fanedge] (p1)--(p3);
\draw[fanedge] (p2)--(p3);

\node[fandot] at (p0) {};
\node[fanlabel] at (1.32500,2.15000) {$C$};
\node[fandot] at (p1) {};
\node[fanlabel] at (2.79000,-0.12000) {$B$};
\node[fandot] at (p2) {};
\node[fanlabel] at (-0.12000,-0.12000) {$A$};
\node[fandot] at (p3) {};
\node[fanlabel] at (1.56000,-1.92000) {$D$};

\node at (barycentric cs:p0=1,p1=1,p3=1) {$a$};
\node at (barycentric cs:p0=1,p2=1,p3=1) {$b$};

\end{scope}
\node[paneltitle] at (7.32500,0) {$\Sigma'$};
\begin{scope}[shift={(0.000,-4.850)}]
\coordinate (p0) at (1.32500,2.00000); 
\coordinate (p1) at (2.65000,0.00000); 
\coordinate (p2) at (0.00000,0.00000); 
\coordinate (p3) at (1.70000,-1.75000); 
\coordinate (p4) at (1.32500,-0.68000); 

\draw[fanedge] (p0)--(p1);
\draw[fanedge] (p0)--(p2);
\draw[fanedge] (p1)--(p2);
\draw[fanedge] (p1)--(p3);
\draw[fanedge] (p1)--(p4);
\draw[fanedge] (p2)--(p3);
\draw[fanedge] (p2)--(p4);
\draw[fanedge] (p3)--(p4);

\node[fandot] at (p0) {};
\node[fanlabel] at (1.32500,2.15000) {$C$};
\node[fandot] at (p1) {};
\node[fanlabel] at (2.79000,-0.12000) {$B$};
\node[fandot] at (p2) {};
\node[fanlabel] at (-0.12000,-0.12000) {$A$};
\node[fandot] at (p3) {};
\node[fanlabel] at (1.56000,-1.92000) {$D$};
\node[fandot] at (p4) {};
\node[fanlabel] at (1.32500,-0.50000) {$V$};

\node[fanlabel] at (barycentric cs:p2=1,p4=1,p3=1) {$a$};
\node[fanlabel] at (barycentric cs:p1=1,p4=1,p3=1) {$b$};

\node[paneltitle] at (1.32500,-2.35000) {$\Sigma_1$};
\end{scope}

\begin{scope}[shift={(4.100,-4.850)}]
\coordinate (p0) at (1.32500,2.00000); 
\coordinate (p1) at (2.65000,0.00000); 
\coordinate (p2) at (0.00000,0.00000); 
\coordinate (p3) at (1.70000,-1.75000); 
\coordinate (p4) at (1.32500,-0.68000); 

\draw[fanedge] (p0)--(p1);
\draw[fanedge] (p0)--(p2);
\draw[fanedge] (p0)--(p4);
\draw[fanedge] (p1)--(p3);
\draw[fanedge] (p1)--(p4);
\draw[fanedge] (p2)--(p3);
\draw[fanedge] (p2)--(p4);
\draw[fanedge] (p3)--(p4);

\node[fandot] at (p0) {};
\node[fanlabel] at (1.32500,2.15000) {$C$};
\node[fandot] at (p1) {};
\node[fanlabel] at (2.79000,-0.12000) {$B$};
\node[fandot] at (p2) {};
\node[fanlabel] at (-0.12000,-0.12000) {$A$};
\node[fandot] at (p3) {};
\node[fanlabel] at (1.56000,-1.92000) {$D$};
\node[fandot] at (p4) {};
\node[fanlabel] at (1.32500,-0.50000) {$V$};

\node[fanlabel] at (barycentric cs:p0=1,p2=1,p4=1) {$1$};
\node[fanlabel] at (barycentric cs:p0=1,p1=1,p4=1) {$1$};
\node[fanlabel] at (barycentric cs:p1=1,p4=1,p3=1) {$b$};

\node[paneltitle] at (1.32500,-2.35000) {$\Sigma_2$};
\end{scope}

\begin{scope}[shift={(8.200,-4.850)}]
\coordinate (p0) at (1.32500,2.00000); 
\coordinate (p1) at (2.65000,0.00000); 
\coordinate (p2) at (0.00000,0.00000); 
\coordinate (p3) at (1.70000,-1.75000); 
\coordinate (p4) at (1.32500,-0.68000); 

\draw[fanedge] (p0)--(p1);
\draw[fanedge] (p0)--(p2);
\draw[fanedge] (p0)--(p3);
\draw[fanedge] (p0)--(p4);
\draw[fanedge] (p1)--(p3);
\draw[fanedge] (p2)--(p3);
\draw[fanedge] (p2)--(p4);
\draw[fanedge] (p3)--(p4);

\node[fandot] at (p0) {};
\node[fanlabel] at (1.32500,2.15000) {$C$};
\node[fandot] at (p1) {};
\node[fanlabel] at (2.79000,-0.12000) {$B$};
\node[fandot] at (p2) {};
\node[fanlabel] at (-0.12000,-0.12000) {$A$};
\node[fandot] at (p3) {};
\node[fanlabel] at (1.56000,-1.92000) {$D$};
\node[fandot] at (p4) {};
\node[fanlabel] at (1.32500,-0.50000) {$V$};

\node[fanlabel] at (barycentric cs:p0=1,p3=1,p4=1) {$b-a$};
\node[fanlabel] at (barycentric cs:p1=1,p0=1,p3=1) {$a$};

\node[paneltitle] at (1.32500,-2.35000) {$\Sigma_3$};
\end{scope}


\draw[flowarrow] (4.95,0.00) -- (6.20,0.00);

\draw[flowarrow] (2.325,-2) -- (2.325,-3);      
\draw[flowarrow] (3.10,-3.85) -- (4.10,-3.85);        
\draw[flowarrow] (7.20,-3.85) -- (8.20,-3.85);        
\draw[flowarrow] (8.525,-3) -- (8.525,-2);      
\end{tikzpicture}
\caption{The case $a<b$}
\label{fig:I-flip a<b}
\end{figure}

If $a>b$, we can interchange $A,B$ and apply the same argument with index $a<r$.

If $a=b$, then we have $a=b=1$, $r=2$ since $\gcd(a,b)=1$, so $ACD, BCD$ are regular. After the first flop, the cones in (\ref{eq: I flip Sigma_2}) are regular and $V=C+D$. By the classification \ref{subsec: toric blow-down and flop} and \ref{subsec: toric mmp}, $X_{\Sigma_2}\to X_{\Sigma'}$ is a smooth curve blow-down by inserting $V$ in $CD$. Then the rest follows by Lemma \ref{lem: lift curve blow-down and pt blow-down to sing pt}.

\begin{figure}[htbp]
\centering
\begin{tikzpicture}[x=1cm,y=1cm]

\tikzset{
  flowarrow/.style={->, dashed, line width=0.5pt},
  paneltitle/.style={font=\small}
}

\begin{scope}[shift={(0.000,-4.850)}]
\coordinate (p0) at (1.32500,2.00000); 
\coordinate (p1) at (2.65000,0.00000); 
\coordinate (p2) at (0.00000,0.00000); 
\coordinate (p3) at (1.32500,-1.75000); 
\coordinate (p4) at (1.32500,-0.68000); 

\draw[fanedge] (p0)--(p1);
\draw[fanedge] (p0)--(p2);
\draw[fanedge] (p1)--(p2);
\draw[fanedge] (p1)--(p3);
\draw[fanedge] (p1)--(p4);
\draw[fanedge] (p2)--(p3);
\draw[fanedge] (p2)--(p4);
\draw[fanedge] (p3)--(p4);

\node[fandot] at (p0) {};
\node[fanlabel] at (1.32500,2.15000) {$C$};
\node[fandot] at (p1) {};
\node[fanlabel] at (2.79000,-0.12000) {$B$};
\node[fandot] at (p2) {};
\node[fanlabel] at (-0.12000,-0.12000) {$A$};
\node[fandot] at (p3) {};
\node[fanlabel] at (1.56000,-1.92000) {$D$};
\node[fandot] at (p4) {};
\node[fanlabel] at (1.32500,-0.50000) {$V$};

\node[fanlabel] at (barycentric cs:p2=1,p4=1,p3=1) {$1$};
\node[fanlabel] at (barycentric cs:p1=1,p4=1,p3=1) {$1$};

\node[paneltitle] at (1.32500,-2.35000) {$\Sigma_1$};
\end{scope}

\begin{scope}[shift={(4.100,-4.850)}]
\coordinate (p0) at (1.32500,2.00000); 
\coordinate (p1) at (2.65000,0.00000); 
\coordinate (p2) at (0.00000,0.00000); 
\coordinate (p3) at (1.32500,-1.75000); 
\coordinate (p4) at (1.32500,-0.68000); 

\draw[fanedge] (p0)--(p1);
\draw[fanedge] (p0)--(p2);
\draw[fanedge] (p0)--(p4);
\draw[fanedge] (p1)--(p3);
\draw[fanedge] (p1)--(p4);
\draw[fanedge] (p2)--(p3);
\draw[fanedge] (p2)--(p4);
\draw[fanedge] (p3)--(p4);

\node[fandot] at (p0) {};
\node[fanlabel] at (1.32500,2.15000) {$C$};
\node[fandot] at (p1) {};
\node[fanlabel] at (2.79000,-0.12000) {$B$};
\node[fandot] at (p2) {};
\node[fanlabel] at (-0.12000,-0.12000) {$A$};
\node[fandot] at (p3) {};
\node[fanlabel] at (1.56000,-1.92000) {$D$};
\node[fandot] at (p4) {};
\node[fanlabel] at (1.32500,-0.50000) {$V$};

\node[fanlabel] at (barycentric cs:p0=1,p2=1,p4=1) {$1$};
\node[fanlabel] at (barycentric cs:p0=1,p1=1,p4=1) {$1$};

\node[paneltitle] at (1.32500,-2.35000) {$\Sigma_2$};
\end{scope}

\begin{scope}[shift={(8.200,-4.850)}]
\coordinate (p0) at (1.32500,2.00000); 
\coordinate (p1) at (2.65000,0.00000); 
\coordinate (p2) at (0.00000,0.00000); 
\coordinate (p3) at (1.32500,-1.75000); 
\coordinate (p4) at (1.32500,-0.68000); 

\draw[fanedge] (p0)--(p1);
\draw[fanedge] (p0)--(p2);
\draw[fanedge] (p1)--(p3);
\draw[fanedge] (p0)--(p3);
\draw[fanedge] (p2)--(p3);

\node[fandot] at (p0) {};
\node[fanlabel] at (1.32500,2.15000) {$C$};
\node[fandot] at (p1) {};
\node[fanlabel] at (2.79000,-0.12000) {$B$};
\node[fandot] at (p2) {};
\node[fanlabel] at (-0.12000,-0.12000) {$A$};
\node[fandot] at (p3) {};
\node[fanlabel] at (1.56000,-1.92000) {$D$};

\node[fanlabel] at (barycentric cs:p0=1,p2=1,p3=1) {$1$};
\node[fanlabel] at (barycentric cs:p0=1,p1=1,p3=1) {$1$};

\node[paneltitle] at (1.32500,-2.35000) {$\Sigma'$};
\end{scope}



\draw[flowarrow] (3.10,-3.85) -- (4.10,-3.85);        
\draw[flowarrow] (7.20,-3.85) -- (8.20,-3.85);        

\end{tikzpicture}
\caption{The case $a=b$}
\label{fig:I-flip, a=b}
\end{figure}

The readers could easily show that exactly one curve blowdown and at most $r-1$ flops occur in this process.    
\end{proof}

\begin{lem}\label{lem: lift II flip}
Proposition \ref{prop: lifting mmp} holds in Case 3, Family (II) of \ref{subsec: toric mmp}.    
\end{lem}

\begin{proof}
We will prove this with induction on $r$. As the first step, we want to find and insert the mld-ray $V$ of $ABD$. for this we write 
\begin{align*}
N =&\Zz A+\Zz B+\Zz D+\Zz C \\
  =&\Zz A+\Zz B+\Zz D+\Zz\frac{1}{r}(aA+B-D)\\
  =&\Zz A+\Zz B+\Zz D+\Zz\frac{1}{r}(akA+kB+(r-k)D), \ \forall \gcd(k,r)=1.   
\end{align*}
Since $\gcd(a,r)=1$, we can choose $k\in\{1,\dots,r-1\}$ such that $ak\equiv1\mod r$, then $N=\Zz A+\Zz B+\Zz D+\Zz\cdot\frac{1}{r}(A+kB+(r-k)D)$ and one sees 
\[
rV=A+kB+sD,\ \ s:=r-k
\]
Inserting $V$ gives us a subdivision $\Sigma_1$ and the relevant
\[
ABC,ABV,ADV,BDV\ \ \ \text{with indices }1,s,k,1.
\]
By Definition \ref{def: canonical resolution} again we have $\Sigma^c=\Sigma_1^c$.
To know what operations we shall apply, we also write $V$ in other coordinates $\{A,B,C\},\{A,C,D\},\{C,B,D\}$. 

Since $D=aA+B-rC$, we have
\begin{align*}
rV&=A+kB+sD=(sa+1)A+(k+s)B-srC\\
V&=pA+B-sC,\ p:=\frac{sa+1}{r}\in\Zz_{>0}
\end{align*}
It is easy to see that
\[\gcd(p,s)=1,\ \ \ \gcd(p,a)=1,  \ \ \ p\le a,s<r.\]
If $p=s$, then $a=r-1$ and $s=p=1$. If $p=a$, then $s=r-1$ and $p=a=k=1$.

Replace $B$ with $D+rC-aA$ in $V=pA+B-sC$, we get
\[
V+(a-p)A=D+(r-s)C=D+kC.
\]
Since $a-p=\frac{ar-as-1}{r}=\frac{ak-1}{r}$, we have $\gcd(a-p,k)=1$ and $0\le a-p<k$.

Replace $A$ with $rV-kB-sD$ in $D=aA+B-rC$, we get 
\begin{align*}
D&=a(rV-kB-sD)+B-rC   \\
arV&=(as+1)D+(ak-1)B+rC\\
aV&=pD+(a-p)B+C.
\end{align*}
Now the picture is clear, we first change the diagonal $AB$ to $CV$ and get
\[
\Sigma_1\rightsquigarrow \Sigma_2\ ,\ \{ABC, ABV\}\rightsquigarrow\{ACV,BCV\}
\]
as in \ref{subsec: toric blow-down and flop}. 
\begin{figure}[htbp]
\centering
\begin{tikzpicture}[x=1cm,y=1cm]

\tikzset{
  flowarrow/.style={->, dashed, line width=0.5pt},
  paneltitle/.style={font=\small}
}

\begin{scope}[shift={(0.000,0)}]
\coordinate (p0) at (1.32500,2.00000); 
\coordinate (p1) at (2.65000,0.00000); 
\coordinate (p2) at (0.00000,0.00000); 
\coordinate (p3) at (1.0000,-1.75000); 

\draw[fanedge] (p0)--(p1);
\draw[fanedge] (p0)--(p2);
\draw[fanedge] (p1)--(p2);
\draw[fanedge] (p1)--(p3);
\draw[fanedge] (p2)--(p3);

\node[fandot] at (p0) {};
\node[fanlabel] at (1.32500,2.15000) {$C$};
\node[fandot] at (p1) {};
\node[fanlabel] at (2.79000,-0.12000) {$B$};
\node[fandot] at (p2) {};
\node[fanlabel] at (-0.12000,-0.12000) {$A$};
\node[fandot] at (p3) {};
\node[fanlabel] at (1.14000,-1.92000) {$D$};

\node[fanlabel] at (barycentric cs:p0=1,p1=1,p2=1) {$1$};
\node[fanlabel] at (barycentric cs:p1=1,p2=1,p3=1) {$r$};

\node[paneltitle] at (1.32500,-2.35000) {$\Sigma$};
\end{scope}

\begin{scope}[shift={(4.100,0)}]
\coordinate (p0) at (1.32500,2.00000); 
\coordinate (p1) at (2.65000,0.00000); 
\coordinate (p2) at (0.00000,0.00000); 
\coordinate (p3) at (1.0000,-1.75000); 
\coordinate (p4) at (1.32500,-0.68000); 

\draw[fanedge] (p0)--(p1);
\draw[fanedge] (p0)--(p2);
\draw[fanedge] (p1)--(p2);
\draw[fanedge] (p1)--(p3);
\draw[fanedge] (p1)--(p4);
\draw[fanedge] (p2)--(p3);
\draw[fanedge] (p2)--(p4);
\draw[fanedge] (p3)--(p4);

\node[fandot] at (p0) {};
\node[fanlabel] at (1.32500,2.15000) {$C$};
\node[fandot] at (p1) {};
\node[fanlabel] at (2.79000,-0.12000) {$B$};
\node[fandot] at (p2) {};
\node[fanlabel] at (-0.12000,-0.12000) {$A$};
\node[fandot] at (p3) {};
\node[fanlabel] at (1.14000,-1.92000) {$D$};
\node[fandot] at (p4) {};
\node[fanlabel] at (1.32500,-0.50000) {$V$};

\node[fanlabel] at (barycentric cs:p2=1,p4=1,p3=1) {$k$};
\node[fanlabel] at (barycentric cs:p1=1,p4=1,p3=1) {$1$};
\node[fanlabel] at (barycentric cs:p2=1,p4=1,p1=1) {$s$};

\node[paneltitle] at (1.32500,-2.35000) {$\Sigma_1$};
\end{scope}

\begin{scope}[shift={(8.200,0)}]
\coordinate (p0) at (1.32500,2.00000); 
\coordinate (p1) at (2.65000,0.00000); 
\coordinate (p2) at (0.00000,0.00000); 
\coordinate (p3) at (1.0000,-1.75000); 
\coordinate (p4) at (1.32500,-0.68000); 

\draw[fanedge] (p0)--(p1);
\draw[fanedge] (p0)--(p2);
\draw[fanedge] (p0)--(p4);
\draw[fanedge] (p1)--(p3);
\draw[fanedge] (p1)--(p4);
\draw[fanedge] (p2)--(p3);
\draw[fanedge] (p2)--(p4);
\draw[fanedge] (p3)--(p4);

\node[fandot] at (p0) {};
\node[fanlabel] at (1.32500,2.15000) {$C$};
\node[fandot] at (p1) {};
\node[fanlabel] at (2.79000,-0.12000) {$B$};
\node[fandot] at (p2) {};
\node[fanlabel] at (-0.12000,-0.12000) {$A$};
\node[fandot] at (p3) {};
\node[fanlabel] at (1.14000,-1.92000) {$D$};
\node[fandot] at (p4) {};
\node[fanlabel] at (1.32500,-0.50000) {$V$};

\node[fanlabel] at (barycentric cs:p2=1,p4=1,p3=1) {$k$};
\node[fanlabel] at (barycentric cs:p1=1,p4=1,p3=1) {$1$};
\node[fanlabel] at (barycentric cs:p2=1,p4=1,p0=1) {$1$};
\node[fanlabel] at (barycentric cs:p1=1,p4=1,p0=1) {$p$};

\node[paneltitle] at (1.32500,-2.35000) {$\Sigma_2$};
\end{scope}



\draw[flowarrow] (3.10,1) -- (4.10,1);        
\draw[flowarrow] (7.20,1) -- (8.20,1);        

\end{tikzpicture}
\label{fig:II-flip }
\end{figure}

If $p<s$, the relation
$V=pA+B-sC$ implies that $X_{\Sigma_1}\dasharrow X_{\Sigma_2}$ is a flip of type $\mathrm{II}(s,p)$. Since $s<r$, by the induction we know Proposition \ref{prop: lifting mmp} holds for $X_{\Sigma_1}\dasharrow X_{\Sigma_2}$. If $p=s=1$, we have $A+B=C+V$ and $X_{\Sigma_1}\dasharrow X_{\Sigma_2}$ is a smooth $K$-flop which lifts to $X_{\Sigma_1^c}\dasharrow X_{\Sigma_2^c}$ by Definition \ref{def: canonical resolution}. Hence it suffices to show $X_{\Sigma_2^c}\dasharrow X_{(\Sigma')^c}$ has the desired decomposition. 

In either case the relevant maximal cones in $\Sigma_2$ are
\begin{equation}\label{eq: II flip sigma_2}
 ACV,BCV,ADV,BDV\ \ \ \text{with indices }1,p,k,1.
\end{equation}

If $a=p$, then $a=p=k=1$ and $V=C+D$. Every cone of
(\ref{eq: II flip sigma_2}) is regular, so $X_{\Sigma_2}\to X_{\Sigma'}$ is the smooth curve blow-down as described in \ref{subsec: toric blow-down and flop}, which lifts to their canonical resolutions $X_{\Sigma^c_2}\to X_{(\Sigma')^c}$ by Definition \ref{def: canonical resolution}. 

If $a>p$, $ACV$ is regular and the relations
\[
D=(a-p)A+V-kC\ ,\ 0<a-p<k
\]
implies that $X_{\Sigma_2}\dasharrow X_{\Sigma_3}$ is a flip of type $\mathrm{II}(k,a-p)$, where $\Sigma_2\rightsquigarrow\Sigma_3$ is given by 
\[
\{ACV,ADV\}\rightsquigarrow\{ACD,CDV\}.
\]
Since $k<r$, by induction Proposition \ref{prop: lifting mmp} holds for $X_{\Sigma_2}\dasharrow X_{\Sigma_3}$, then it suffices to show that $X_{\Sigma_3^c}=X_{(\Sigma')^c}$.

The relevant maximal cones in $\Sigma_3$ are
\[
ACD,CDV,BCV,BDV\ \ \ \text{with indices }1,a-p,p,1.
\]
The relations 
\[
aV=(a-p)B+C+pD\ ,\  a>p\ge 1
\]
imply that $V$ is precisely the mld-ray in the index $a$ singular terminal cone $BCD$.
By Definition \ref{def: canonical resolution} again we have $\Sigma^c_3=(\Sigma')^c$.

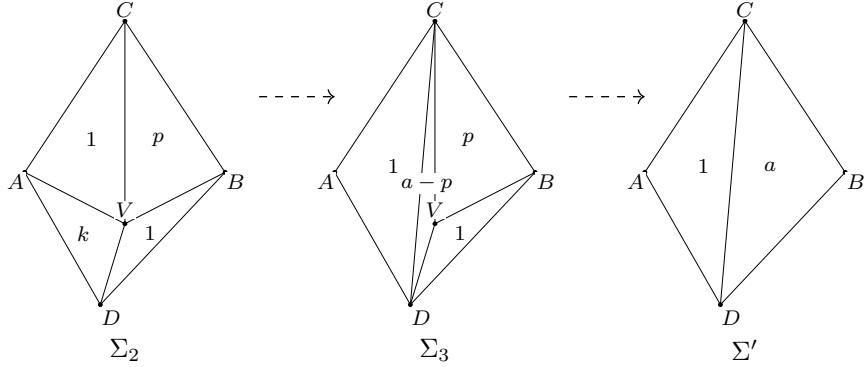
\begin{figure}[htbp]
\centering
\begin{tikzpicture}[x=1cm,y=1cm]

\tikzset{
  flowarrow/.style={->, dashed, line width=0.5pt},
  paneltitle/.style={font=\small}
}

\begin{scope}[shift={(0.000,0)}]
\coordinate (p0) at (1.32500,2.00000); 
\coordinate (p1) at (2.65000,0.00000); 
\coordinate (p2) at (0.00000,0.00000); 
\coordinate (p3) at (1.0000,-1.75000); 
\coordinate (p4) at (1.32500,-0.68000); 

\draw[fanedge] (p0)--(p1);
\draw[fanedge] (p0)--(p2);
\draw[fanedge] (p0)--(p4);
\draw[fanedge] (p1)--(p3);
\draw[fanedge] (p1)--(p4);
\draw[fanedge] (p2)--(p3);
\draw[fanedge] (p2)--(p4);
\draw[fanedge] (p3)--(p4);

\node[fandot] at (p0) {};
\node[fanlabel] at (1.32500,2.15000) {$C$};
\node[fandot] at (p1) {};
\node[fanlabel] at (2.79000,-0.12000) {$B$};
\node[fandot] at (p2) {};
\node[fanlabel] at (-0.12000,-0.12000) {$A$};
\node[fandot] at (p3) {};
\node[fanlabel] at (1.14000,-1.92000) {$D$};
\node[fandot] at (p4) {};
\node[fanlabel] at (1.32500,-0.50000) {$V$};

\node[fanlabel] at (barycentric cs:p2=1,p4=1,p3=1) {$k$};
\node[fanlabel] at (barycentric cs:p1=1,p4=1,p3=1) {$1$};
\node[fanlabel] at (barycentric cs:p2=1,p4=1,p0=1) {$1$};
\node[fanlabel] at (barycentric cs:p1=1,p4=1,p0=1) {$p$};

\node[paneltitle] at (1.32500,-2.35000) {$\Sigma_2$};
\end{scope}

\begin{scope}[shift={(4.100,0)}]
\coordinate (p0) at (1.32500,2.00000); 
\coordinate (p1) at (2.65000,0.00000); 
\coordinate (p2) at (0.00000,0.00000); 
\coordinate (p3) at (1.0000,-1.75000); 
\coordinate (p4) at (1.32500,-0.68000); 

\draw[fanedge] (p0)--(p1);
\draw[fanedge] (p0)--(p2);
\draw[fanedge] (p0)--(p3);
\draw[fanedge] (p0)--(p4);
\draw[fanedge] (p1)--(p3);
\draw[fanedge] (p1)--(p4);
\draw[fanedge] (p2)--(p3);
\draw[fanedge] (p3)--(p4);

\node[fandot] at (p0) {};
\node[fanlabel] at (1.32500,2.15000) {$C$};
\node[fandot] at (p1) {};
\node[fanlabel] at (2.79000,-0.12000) {$B$};
\node[fandot] at (p2) {};
\node[fanlabel] at (-0.12000,-0.12000) {$A$};
\node[fandot] at (p3) {};
\node[fanlabel] at (1.14000,-1.92000) {$D$};
\node[fandot] at (p4) {};
\node[fanlabel] at (1.32500,-0.50000) {$V$};

\node[fanlabel] at (barycentric cs:p2=1,p0=1,p3=1) {$1$};
\node[fanlabel] at (barycentric cs:p0=1,p4=1,p3=1) {$a-p$};
\node[fanlabel] at (barycentric cs:p3=1,p4=1,p1=1) {$1$};
\node[fanlabel] at (barycentric cs:p1=1,p4=1,p0=1) {$p$};

\node[paneltitle] at (1.32500,-2.35000) {$\Sigma_3$};
\end{scope}

\begin{scope}[shift={(8.200,0)}]
\coordinate (p0) at (1.32500,2.00000); 
\coordinate (p1) at (2.65000,0.00000); 
\coordinate (p2) at (0.00000,0.00000); 
\coordinate (p3) at (1.0000,-1.75000); 

\draw[fanedge] (p0)--(p1);
\draw[fanedge] (p0)--(p2);
\draw[fanedge] (p0)--(p3);
\draw[fanedge] (p1)--(p3);

\draw[fanedge] (p2)--(p3);

\node[fandot] at (p0) {};
\node[fanlabel] at (1.32500,2.15000) {$C$};
\node[fandot] at (p1) {};
\node[fanlabel] at (2.79000,-0.12000) {$B$};
\node[fandot] at (p2) {};
\node[fanlabel] at (-0.12000,-0.12000) {$A$};
\node[fandot] at (p3) {};
\node[fanlabel] at (1.14000,-1.92000) {$D$};

\node[fanlabel] at (barycentric cs:p2=1,p0=1,p3=1) {$1$};
\node[fanlabel] at (barycentric cs:p1=1,p0=1,p3=1) {$a$};

\node[paneltitle] at (1.32500,-2.35000) {$\Sigma'$};
\end{scope}



\draw[flowarrow] (3.10,1) -- (4.10,1);        
\draw[flowarrow] (7.20,1) -- (8.20,1);        

\end{tikzpicture}
\caption{The case $a>p$}
\label{fig:II-flip a>p }
\end{figure}

Using inductions we can also show that our constructed decomposition in Proposition \ref{prop: lifting mmp} for a flip of type $\mathrm{II}(r,a)$ has exactly $a$ flops and $r-a$ curve blow-downs.   
\end{proof}

\begin{lem}\label{lem: lift pt blow-down to smooth pt}
Proposition \ref{prop: lifting mmp} holds in Case 2 (2) of \ref{subsec: toric mmp}.      
\end{lem}

\begin{proof}
Recall that our $ABC$ is regular, let $W=u=A+aB+bC$ where $a,b>0$, $\gcd(a,b)=1$. $\Sigma$ is the star subdivision of $\Sigma'$ at $W\in ABC$ as in \ref{subsec: toric blow-down and flop}. Then $ABW,BCW,ACW$ are terminal and have indices $1,a,b$.

We use induction on $a+b$, interchanging $B,C$ so that $a\leq b$. If $a=b$, then $a=b=\gcd(a,b)=1$. Then $W=A+B+C$ and $X_\Sigma\to X_{\Sigma'}$ is the
smooth point blow-down which lifts to $X_{\Sigma^c}\to X_{(\Sigma')^c}$.

Suppose $a=1<b$. Let $V=A+(b-1)B+C$, we see
\[
bV=A+C+(b-1)W,
\]
so $V$ is the mld-ray in $ACW$. Let $\Sigma_1$ be the subdivision of $\Sigma$ by inserting $V$, then by Definition \ref{def: canonical resolution} we have $\Sigma_1^c=\Sigma^c$. Notice that $W=V+B$ and $BCV,ABV$ are regular, this implies that 
\[
\Sigma_1 \rightsquigarrow \Sigma_3\ ,\  \{AVW,ABW,BCW,CVW\}\rightsquigarrow \{BCV,ABV\}
\]
is the star subdivision such that $X_{\Sigma_1}\to X_{\Sigma_3}$ is a smooth curve blow-down as in \ref{subsec: toric blow-down and flop}, where $\Sigma_3$ is obtained from $\Sigma_1$ by deleting $W$. By Definition \ref{def: canonical resolution} again $X_{\Sigma_1}\to X_{\Sigma_3}$ lifts to a smooth curve blow-down $X_{\Sigma_1^c}\to X_{\Sigma_3^c}$. Therefore it suffices to show $X_{\Sigma_3^c}\dasharrow X_{(\Sigma')^c}$ has the desired decompositions.
\begin{figure}[htbp]
\centering
\begin{tikzpicture}

\tikzset{
  flowarrow/.style={->, dashed, line width=0.5pt},
  paneltitle/.style={font=\small}
}

\begin{scope}[shift={(0.000,-3.300)}]
\coordinate (p0) at (1.32500,2.29497); 
\coordinate (p1) at (2.65000,0.00000); 
\coordinate (p2) at (0.00000,0.00000); 
\coordinate (p4) at (1.98750,0.38255); 
\draw[fanedge] (p0)--(p1);
\draw[fanedge] (p0)--(p2);
\draw[fanedge] (p0)--(p4);
\draw[fanedge] (p1)--(p2);
\draw[fanedge] (p2)--(p4);
\draw[fanedge] (p1)--(p4);
\node[fandot] at (p0) {};
\node[fanlabel] at (1.32500,2.43497) {$C$};
\node[fandot] at (p1) {};
\node[fanlabel] at (2.76000,-0.13000) {$B$};
\node[fandot] at (p2) {};
\node[fanlabel] at (-0.11000,-0.13000) {$A$};
\node[fandot] at (p4) {};
\node[fanlabel] at (1.73000,0.59) {$W$};
\node[fanlabel] at (barycentric cs:p2=1,p4=1,p1=1) {$1$};
\node[fanlabel] at (barycentric cs:p1=1,p0=1,p4=1) {$1$};
\node[fanlabel] at (barycentric cs:p2=1,p4=1,p0=1) {$b$};

\node[paneltitle] at (1.325,-0.440) {$\Sigma$};
\end{scope}
\begin{scope}[shift={(3.800,-3.300)}]
\coordinate (p0) at (1.32500,2.29497); 
\coordinate (p1) at (2.65000,0.00000); 
\coordinate (p2) at (0.00000,0.00000); 
\coordinate (p3) at (1.32500,0.76499); 
\coordinate (p4) at (1.98750,0.38255); 
\draw[fanedge] (p0)--(p1);
\draw[fanedge] (p0)--(p2);
\draw[fanedge] (p0)--(p3);
\draw[fanedge] (p0)--(p4);
\draw[fanedge] (p1)--(p2);
\draw[fanedge] (p2)--(p4);
\draw[fanedge] (p1)--(p4);
\draw[fanedge] (p2)--(p3);
\draw[fanedge] (p3)--(p4);
\node[fandot] at (p0) {};
\node[fanlabel] at (1.32500,2.43497) {$C$};
\node[fandot] at (p1) {};
\node[fanlabel] at (2.76000,-0.13000) {$B$};
\node[fandot] at (p2) {};
\node[fanlabel] at (-0.11000,-0.13000) {$A$};
\node[fandot] at (p3) {};
\node[fanlabel] at (1.18500,0.68499) {$V$};
\node[fandot] at (p4) {};
\node[fanlabel] at (1.73000,0.59) {$W$};

\node[fanlabel] at (barycentric cs:p2=1,p4=1,p3=1) {$1$};
\node[fanlabel] at (barycentric cs:p3=1,p0=1,p4=1) {$1$};
\node[fanlabel] at (barycentric cs:p2=1,p3=1,p0=1) {$b-1$};

\node[paneltitle] at (1.325,-0.440) {$\Sigma_1$};
\end{scope}
\begin{scope}[shift={(7.600,-3.300)}]
\coordinate (p0) at (1.32500,2.29497); 
\coordinate (p1) at (2.65000,0.00000); 
\coordinate (p2) at (0.00000,0.00000); 
\coordinate (p3) at (1.32500,0.76499); 
\draw[fanedge] (p0)--(p1);
\draw[fanedge] (p0)--(p2);
\draw[fanedge] (p0)--(p3);
\draw[fanedge] (p1)--(p2);
\draw[fanedge] (p1)--(p3);
\draw[fanedge] (p2)--(p3);
\node[fandot] at (p0) {};
\node[fanlabel] at (1.32500,2.43497) {$C$};
\node[fandot] at (p1) {};
\node[fanlabel] at (2.76000,-0.13000) {$B$};
\node[fandot] at (p2) {};
\node[fanlabel] at (-0.11000,-0.13000) {$A$};
\node[fandot] at (p3) {};
\node[fanlabel] at (1.18500,0.68499) {$V$};

\node[fanlabel] at (barycentric cs:p2=1,p1=1,p3=1) {$1$};
\node[fanlabel] at (barycentric cs:p3=1,p0=1,p1=1) {$1$};
\node[fanlabel] at (barycentric cs:p2=1,p3=1,p0=1) {$b-1$};

\node[paneltitle] at (1.325,-0.440) {$\Sigma_3$};
\end{scope}

\end{tikzpicture}
\caption{The case $1=a<b$}
\end{figure}
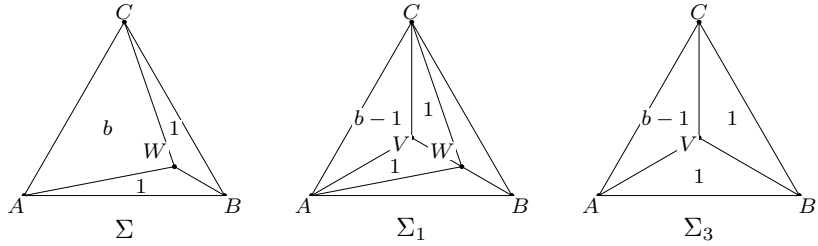

Recall that $V=A+(b-1)B+C$, therefore $X_{\Sigma_3}\to X_{\Sigma'}$ is the weighted $(1,b-1,1)$-blowup in Case 2 (2) of \ref{subsec: toric mmp}. By induction Proposition \ref{prop: lifting mmp} holds for $X_{\Sigma_3}\to X_{\Sigma'}$ and we are done.

Now suppose $1<a<b$. Then we would like to find the mld-ray $V$ of $ACW$:
\begin{align*}
N&=\Zz A+\Zz C+\Zz W+\Zz B \\
 &=\Zz A+\Zz C+\Zz W+\Zz \frac{1}{b}(-aA-C+W)\\
 &=\Zz A+\Zz C+\Zz W+\Zz \frac{1}{b}(-aqA+(b-q)C+qW),\ \forall \gcd(q,b)=1
\end{align*}
Choose $q\in\{1,\dots,b-1\}$ with $-aq\equiv1\mod b$, then we see 
\begin{equation}\label{eq: V=A,C,W}
 bV=A+(b-q)C+qW\ ,\ \gcd(q,b)=1   
\end{equation}
Let $\Sigma_1$ be the subdivision of $\Sigma$ by inserting $V$, then we have $X_{\Sigma_1^c}=X_{\Sigma^c}$ and $CVW,AVW,ACV$ are terminal with indices $1,b-q,q$.

Replacing $A$ with $bV-(b-q)C-qW$ in $W=aA+bB+C$, we get
\[
aV+B=\frac{aq+1}{b}W+(a-\frac{aq+1}{b})C.
\]
Let $p:=\frac{aq+1}{b}<\frac{aq+a}{b}\le \frac{ab}{b}=a$, then $p<a$, $\gcd(a,p)=\gcd(q,p)=1$ and the above equation becomes
\[
pW+(a-p)C=aV+B.
\]
Changing $CW$ to $BV$, we see
\[
\Sigma_1 \rightsquigarrow\Sigma_2\ ,\ \{CVW,BCW\}\rightsquigarrow\{BCV,BVW\}
\]
describes a flip of type $\mathrm{I}(a-p,p)$ as in \ref{subsec: toric blow-down and flop} and \ref{subsec: toric mmp}.

By Lemma \ref{lem: lift I flip} we know Proposition \ref{prop: lifting mmp} holds for $X_{\Sigma_1}\dasharrow X_{\Sigma_2}$. Hence it suffices to show $X_{\Sigma_2^c}\dasharrow X_{(\Sigma')^c}$ has the desired decomposition. 

Replacing $W$ with $aA+bB+C$ in the equation (\ref{eq: V=A,C,W}) we get
\begin{equation}\label{eq: V=A,B,C}
V=\frac{aq+1}{b}A+qB+C=pA+qB+C,    
\end{equation}
in particular $ABV$ is regular. Replacing $C$ with $V-pA-qB$ in equation (\ref{eq: V=A,C,W}) we get 
\[
W=\frac{(b-q)p-1}{q}A+(b-q)B+V=(a-p)A+(b-q)B+V,
\]
where we used $\frac{(b-q)p-1}{q}=a-p$, which also implies $\gcd(a-p,b-q)=1$. Therefore it describes a weighted $(a-p,b-q,1)$-blowup $$X_{\Sigma_2}\to X_{\Sigma_3}$$ in the Case 2 (2) of \ref{subsec: toric mmp}, where $\Sigma_3$ is obtained from $\Sigma_2$ by deleting $W$, replacing $AVW,ABW,BVW$ with $ABV$. By induction Proposition \ref{prop: lifting mmp} holds for $X_{\Sigma_2}\to X_{\Sigma_3}$ so it remains to show $X_{\Sigma^c_3}\dasharrow X_{(\Sigma')^c}$ has the required decomposition.

Now we notice that the relation (\ref{eq: V=A,B,C}) and $\gcd(p,q)=1$ imply that
$X_{\Sigma_3}\to X_{\Sigma'}$ is a weighted $(p,q,1)$-blowup in the Case 2 (2) of \ref{subsec: toric mmp}, then by induction Proposition \ref{prop: lifting mmp} holds for $X_{\Sigma_3}\to X_{\Sigma'}$ and we are done.

\begin{figure}[htbp]
\centering
\begin{tikzpicture}
\begin{scope}[shift={(0.000,0.000)}]
\coordinate (p0) at (1.32500,2.29497); 
\coordinate (p1) at (2.65000,0.00000); 
\coordinate (p2) at (0.00000,0.00000); 
\coordinate (p4) at (1.62500,0.31); 

\draw[fanedge] (p0)--(p1);
\draw[fanedge] (p0)--(p2);
\draw[fanedge] (p0)--(p4);
\draw[fanedge] (p1)--(p2);
\draw[fanedge] (p1)--(p4);
\draw[fanedge] (p2)--(p4);

\node[fandot] at (p0) {};
\node[fanlabel] at (1.32500,2.43497) {$C$};
\node[fandot] at (p1) {};
\node[fanlabel] at (2.76000,-0.13000) {$B$};
\node[fandot] at (p2) {};
\node[fanlabel] at (-0.11000,-0.13000) {$A$};
\node[fandot] at (p4) {};
\node[fanlabel] at (1.400,0.4) {$W$};

\node[fanlabel] at (barycentric cs:p2=1,p4=1,p1=1) {$1$};
\node[fanlabel] at (barycentric cs:p1=1,p0=1,p4=1) {$a$};
\node[fanlabel] at (barycentric cs:p2=1,p4=1,p0=1) {$b$};

\node[paneltitle] at (1.325,-0.440) {$\Sigma$};
\end{scope}

\begin{scope}[shift={(7.600,0.000)}]
\coordinate (p0) at (1.32500,2.29497); 
\coordinate (p1) at (2.65000,0.00000); 
\coordinate (p2) at (0.00000,0.00000); 

\draw[fanedge] (p0)--(p1);
\draw[fanedge] (p0)--(p2);
\draw[fanedge] (p1)--(p2);

\node[fandot] at (p0) {};
\node[fanlabel] at (1.32500,2.43497) {$C$};
\node[fandot] at (p1) {};
\node[fanlabel] at (2.76000,-0.13000) {$B$};
\node[fandot] at (p2) {};
\node[fanlabel] at (-0.11000,-0.13000) {$A$};

\node[fanlabel] at (barycentric cs:p1=1,p0=1,p2=1) {$1$};

\node[paneltitle] at (1.325,-0.440) {$\Sigma'$};
\end{scope}
\begin{scope}[shift={(0.000,-3.300)}]
\coordinate (p0) at (1.32500,2.29497); 
\coordinate (p1) at (2.65000,0.00000); 
\coordinate (p2) at (0.00000,0.00000); 
\coordinate (p3) at (1.200,1.14748); 
\coordinate (p4) at (1.62500,0.31); 

\draw[fanedge] (p0)--(p1);
\draw[fanedge] (p0)--(p2);
\draw[fanedge] (p0)--(p3);
\draw[fanedge] (p0)--(p4);
\draw[fanedge] (p1)--(p2);
\draw[fanedge] (p1)--(p4);
\draw[fanedge] (p2)--(p3);
\draw[fanedge] (p2)--(p4);
\draw[fanedge] (p3)--(p4);

\node[fandot] at (p0) {};
\node[fanlabel] at (1.32500,2.43497) {$C$};
\node[fandot] at (p1) {};
\node[fanlabel] at (2.76000,-0.13000) {$B$};
\node[fandot] at (p2) {};
\node[fanlabel] at (-0.11000,-0.13000) {$A$};
\node[fandot] at (p3) {};
\node[fanlabel] at (1.200,0.9) {$V$};
\node[fandot] at (p4) {};
\node[fanlabel] at (1.400,0.4) {$W$};

\node[fanlabel] at (barycentric cs:p2=1.7,p4=1,p3=1) {$b-q$};
\node[fanlabel] at (barycentric cs:p3=1,p0=1,p4=1) {$1$};
\node[fanlabel] at (barycentric cs:p2=1,p3=1,p0=1) {$q$};

\node[paneltitle] at (1.325,-0.440) {$\Sigma_1$};
\end{scope}
\begin{scope}[shift={(3.800,-3.300)}]
\coordinate (p0) at (1.32500,2.29497); 
\coordinate (p1) at (2.65000,0.00000); 
\coordinate (p2) at (0.00000,0.00000); 
\coordinate (p3) at (1.200,1.14748); 
\coordinate (p4) at (1.62500,0.31); 

\draw[fanedge] (p0)--(p1);
\draw[fanedge] (p0)--(p2);
\draw[fanedge] (p0)--(p3);
\draw[fanedge] (p1)--(p2);
\draw[fanedge] (p1)--(p3);
\draw[fanedge] (p1)--(p4);
\draw[fanedge] (p2)--(p3);
\draw[fanedge] (p2)--(p4);
\draw[fanedge] (p3)--(p4);

\node[fandot] at (p0) {};
\node[fanlabel] at (1.32500,2.43497) {$C$};
\node[fandot] at (p1) {};
\node[fanlabel] at (2.76000,-0.13000) {$B$};
\node[fandot] at (p2) {};
\node[fanlabel] at (-0.11000,-0.13000) {$A$};
\node[fandot] at (p3) {};
\node[fanlabel] at (1.200,0.9) {$V$};
\node[fandot] at (p4) {};
\node[fanlabel] at (1.300,0.4) {$W$};

\node[fanlabel] at (barycentric cs:p1=1,p4=1,p3=1) {$a-p$};
\node[fanlabel] at (barycentric cs:p3=1,p0=1,p1=1) {$p$};

\node[paneltitle] at (1.325,-0.440) {$\Sigma_2$};
\end{scope}
\begin{scope}[shift={(7.600,-3.300)}]
\coordinate (p0) at (1.32500,2.29497); 
\coordinate (p1) at (2.65000,0.00000); 
\coordinate (p2) at (0.00000,0.00000); 
\coordinate (p3) at (1.200,1.14748); 

\draw[fanedge] (p0)--(p1);
\draw[fanedge] (p0)--(p2);
\draw[fanedge] (p0)--(p3);
\draw[fanedge] (p1)--(p2);
\draw[fanedge] (p1)--(p3);
\draw[fanedge] (p2)--(p3);

\node[fandot] at (p0) {};
\node[fanlabel] at (1.32500,2.43497) {$C$};
\node[fandot] at (p1) {};
\node[fanlabel] at (2.76000,-0.13000) {$B$};
\node[fandot] at (p2) {};
\node[fanlabel] at (-0.11000,-0.13000) {$A$};
\node[fandot] at (p3) {};
\node[fanlabel] at (1.200,0.9) {$V$};

\node[fanlabel] at (barycentric cs:p1=1,p2=1,p3=1) {$1$};
\node[fanlabel] at (barycentric cs:p3=1,p0=1,p1=1) {$p$};
\node[fanlabel] at (barycentric cs:p2=1,p3=1,p0=1) {$q$};

\node[paneltitle] at (1.325,-0.440) {$\Sigma_3$};
\end{scope}
\end{tikzpicture}
\caption{The case $1<a<b$}
\label{fig: weighted lift}
\end{figure}
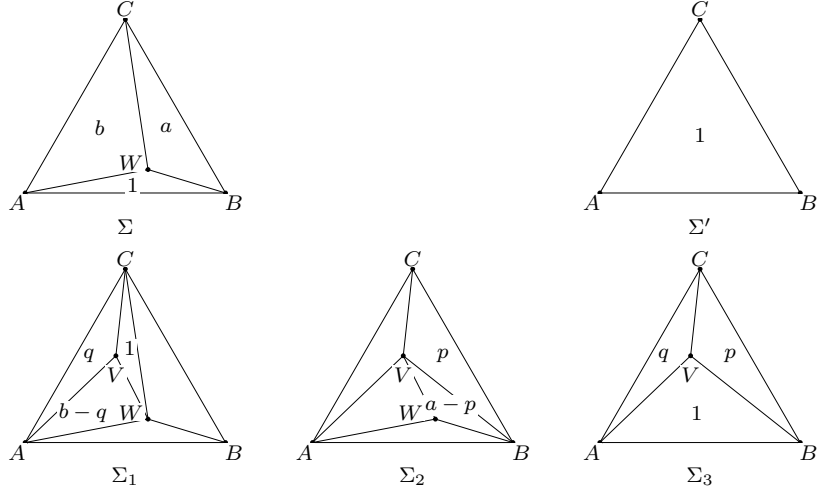
It is easy to check that the number of blow-downs in the decomposion of $X_{\Sigma^c}\dasharrow X_{(\Sigma')^c}$ is $G(\Sigma)+1=a+b-1$. It is interesting to estimate how many flops are needed and we leave it to the readers.
\end{proof}

\section{Projectiveness of intermediate models}

Actually we can show that the intermediate models in Proposition \ref{prop: lifting mmp} are relatively projective over the base of the extremal contractions. This is somehow surprising since in general smooth $K$-flops could destroy projectiveness (like many examples constructed using Atiyah flops).

\begin{prop}\label{prop: rel proj of models}
Let $X_{\Sigma}\dasharrow X_{\Sigma'}$ be as in Proposition \ref{prop: lifting mmp} and $X_{\Sigma}\to X_{\hat\Sigma}$ be the extremal contraction (either the flipping contraction or the divisorial contraction). Then we may assume the $X_i$ are projective over $X_{\hat\Sigma}$.  
\end{prop}

\begin{proof}
Case 1 and Case 2 (1) of \ref{subsec: terminal fans} is obvious. 
Since (relative) projectiveness is transitive and all the local $\Qq$-Cartier divisors could extend, we just need to check the relative projectiveness locally for each case in \ref{subsec: toric mmp}, i.e. we may assume 
\begin{itemize}
    \item $\hat\Sigma=\cone{A,B,C,D}$ in Case 3 of \ref{subsec: toric mmp}.
    \item $\hat\Sigma=ABC$ (and its faces) in Case 2 (2) of \ref{subsec: toric mmp}.
\end{itemize}
The relative projectiveness is equivalent to showing the fan admit strictly convex support functions (linear on each cone), see \cite[Section 6.1]{CLS10} for details.
We will use the figures in Section \ref{sec: sm lifts} to illustrate, and by induction and some standard reduction it suffices to prove it for the following fans:
\begin{itemize}
    \item In the proof of Lemma \ref{lem: lift I flip},
    \[
    \{ADV,BDV,\cone{A,B,C,V}\}\ ;\ \{ACV,ADV,\cone{B,C,D,V}\}
    \]
    \item In the proof of Lemma \ref{lem: lift II flip},
    \[
    \{ADV,BDV,\cone{A,B,C,V}\}\ ;\ \{BCV,BDV,\cone{A,C,D,V}\}
    \]
    \item In the proof of Lemma \ref{lem: lift pt blow-down to smooth pt},
    \[
    \{ACV,AVW,ABW,\cone{B,C,V,W}\}
    \]
\end{itemize}
The existence of the strictly convex support functions can be easily checked for those fans since there are few rays to create any obstruction, we leave the detail to readers for exercise.
\end{proof}

\begin{proof}[Proof of Theorem \ref{thm: decomp bir morph btw toric 3-folds}]
As mentioned in \ref{subsec: toric mmp}, we can run a toric $K_X$-MMP over $Y$ such that it terminates with $Y$, which preserves terminality. Then the decomposition part is given by Proposition \ref{prop: lifting mmp} since the canonical resolutions of $X$ and $Y$ are themselves by Definition \ref{def: canonical resolution}. The projectiveness claim then follows from Proposition \ref{prop: rel proj of models}.    
\end{proof}

\section{Counter-example}

The following example shows that Theorem \ref{thm: decomp bir morph btw toric 3-folds} fails without projectiveness, although the elementary toric $K$-flops and smooth blow-downs could be translated into pure and clean combinatoric languages. Therefore one shall not expect a direct combinatoric proof in the toric setting.

\begin{ex}\label{ex: counter-ex when not proj}
 Consider the ten primitive rays $v_i$
\begin{equation*}
\begin{array}{c|ccccc}
 i&1&2&3&4&5\\\hline
 v_i=(v_{i1},v_{i2},v_{i3})&(1,0,0)&(0,1,0)&(0,0,1)&(1,1,1)&(1,1,2)\\[2pt]
 &6&7&8&9&0\\\hline
 &(1,2,1)&(2,1,1)&(1,1,3)&(2,1,2)&(3,1,2)
\end{array}
\end{equation*}
Let $\Sigma$ be the fan containing the following fifteen maximal cones and all their faces:
\begin{equation*}
\begin{gathered}
127,138,170,180,236,260,270,368,\\
456,459,460,490,568,589,890.
\end{gathered}
\end{equation*}
The indices triple $ijk$ denotes the cone $\cone{v_i,v_j,v_k}$. Using the plane $\{x+y+z=1\}$, we could illustrate $\Sigma$ as the triangulation

\begin{figure}[htbp]
\centering
\begin{tikzpicture}
\begin{scope}[shift={(0.000,0.000)}]
\coordinate (p0) at (2.80000,4.84974);
\coordinate (p1) at (5.60000,0.00000);
\coordinate (p2) at (0.00000,0.00000);
\coordinate (p3) at (2.80000,1.61658);
\coordinate (p4) at (2.80000,2.42487);
\coordinate (p5) at (2.80000,2.90985);
\coordinate (p6) at (3.50000,1.21244);
\coordinate (p7) at (2.10000,1.21244);
\coordinate (p8) at (2.24000,1.93990);
\coordinate (p9) at (1.86667,1.61658);
\draw[fanedge] (p0)--(p1);
\draw[fanedge] (p0)--(p2);
\draw[fanedge] (p0)--(p5);
\draw[fanedge] (p0)--(p6);
\draw[fanedge] (p1)--(p2);
\draw[fanedge] (p1)--(p6);
\draw[fanedge] (p1)--(p7);
\draw[fanedge] (p1)--(p9);
\draw[fanedge] (p2)--(p5);
\draw[fanedge] (p2)--(p7);
\draw[fanedge] (p2)--(p9);
\draw[fanedge] (p3)--(p4);
\draw[fanedge] (p3)--(p6);
\draw[fanedge] (p3)--(p8);
\draw[fanedge] (p3)--(p9);
\draw[fanedge] (p4)--(p5);
\draw[fanedge] (p4)--(p6);
\draw[fanedge] (p4)--(p8);
\draw[fanedge] (p5)--(p6);
\draw[fanedge] (p5)--(p8);
\draw[fanedge] (p5)--(p9);
\draw[fanedge] (p6)--(p9);
\draw[fanedge] (p7)--(p9);
\draw[fanedge] (p8)--(p9);
\node[fandot] at (p0) {};
\node[fanlabel] at (2.80000,5.00974) {$3$};
\node[fandot] at (p1) {};
\node[fanlabel] at (5.72000,-0.13000) {$2$};
\node[fandot] at (p2) {};
\node[fanlabel] at (-0.12000,-0.13000) {$1$};
\node[fandot] at (p3) {};
\node[fanlabel] at (2.95000,1.65658) {$4$};
\node[fandot] at (p4) {};
\node[fanlabel] at (2.95000,2.43487) {$5$};
\node[fandot] at (p5) {};
\node[fanlabel] at (2.94000,2.99985) {$8$};
\node[fandot] at (p6) {};
\node[fanlabel] at (3.65000,1.17244) {$6$};
\node[fandot] at (p7) {};
\node[fanlabel] at (1.95000,1.15244) {$7$};
\node[fandot] at (p8) {};
\node[fanlabel] at (2.10000,1.99990) {$9$};
\node[fandot] at (p9) {};
\node[fanlabel] at (1.69667,1.63658) {$0$};
\end{scope}
\end{tikzpicture}
\caption{The ray points are
$v_i/(v_{i1}+v_{i2}+v_{i3})$ in the affine section plane.}
\label{fig:counter}
\end{figure}
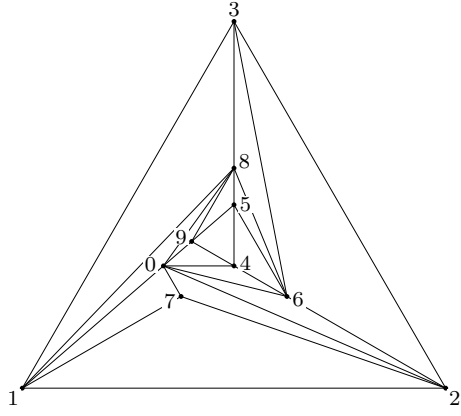  
A direct determinant computations show that every $\sigma\in \Sigma$ is regular.  Consequently $V=X_\Sigma$ is smooth and $V\to\mathbb{A}^3$ is a proper birational morphism since $\Sigma$ is a subdivision of the positive octant. However, as an exercise, the readers can check that
\begin{enumerate}
    \item There is no flopping toric contraction $V\to Z$.
    \item There is no ordinary smooth point or curve blow-down $V\to Z$.
    \item The morphism $V\to\mathbb{A}^3$ is not projective.    
\end{enumerate}
The first two can be done explicitly but (3) is a little bit harder, one needs to deals with the inequalities (ensured by ampleness) between the coefficients of a toric invariant divisor $\sum_id_iD(v_i)$. The readers are also encouraged to think about developing the algorithm to exhaust all the possible fans with rays $v_i\in\Zz^3_{\ge 0}$ and $|v_{i1}+v_{i2}+v_{i3}|\le n$, see how the amount of calculation increases with respect to $n$.
\end{ex}

\noindent\textbf{Acknowledgement and AI disclosure}.

The author first thought about this question around his second year of Ph.D, when he was doing the exercises in Koll\'ar's paper \cite{Kol08}. The author wanted to test the question on toric varieties but it turned out to be subtle, the author managed to write some examples by hands, showing that sufficiently many flops are needed before a smooth blow-down, but was far away from a positive proof or a counter-example. The author would like to thank his advisor Christopher Hacon, for writing an email to Prof. Koll\'ar on whether there is any progress for this question (even for the toric version), the answer was no. He also tried to write a computer program to find a counter-example but he is not good at coding, realizing the amount of calculations grows so rapidly, he considered this way hopeless. He also wants to thank Jingjun Han, Chen Jiang and Felipe Castellano Mac\'ias for some useful discussions.

For some obvious reason, the author becomes interested in this question again just recently. He feels AI should be good at this kind of question and is curious how/if AI could solve it. However, no result was made and no counter-example was found in the past two months, probably because the author can only afford Chat GPT Plus plan, it stopped every (approximately) 30 minutes. Hence he asked his friends who have the Pro plan to try, still no result was made, he was told that several 4h-5h attempts were carried out, even after the new Astra version was released. After that the author is convinced this is a challenging problem that worth sitting down and trying to make a difference as a human being. The author started to read and digest the partial answers carefully, from which he benefits a lot, he also realizes where the essential difficulties are from the interactive communications. 

The turning point comes when the AI mentions a paper \cite{Ch24} in an attempt to solve the general question. The paper is about some weak version minimal resolutions and factorizations for threefolds, it is not enough for our request. But the statements there remind the author that terminal quotient singularity (3-fold) has a canonical resolution, he immediately realizes that this might be enough to solve the toric version by using MMP as explained in the strategy below Theorem \ref{thm: decomp bir morph btw toric 3-folds}, because according to the author's experience in constructing terminal examples (\cite[Theorem 5.7]{LX23}), it is possible to classify the toric terminal blow-downs and flips completely and the lifting process is extremely promising. Therefore the author asks the AI to check if this strategy works and the affirmation is given within 12 minutes. This is surprising since so many previous attempts were made for the toric version as well. It turns out that the author has been asking the wrong question for toric varieties! \emph{The flop is a very simple operation in the toric language, but it loses projectiveness very easily, for induction reason, the author considers proper toric varieties instead of projective ones, which seems more compatible if we want to prove it combinatorially.} However the MMP argument only works for projective varieties. 

So how about the proper case? The author asked the AI to check the toric statement on examples in previous conversations, as expected the computations explode when the height of the rays grows,
\[
h(v)=|v_1|+|v_2|+|v_3|\ ,\ v=(v_1,v_2,v_3)\in \Zz^3_{\ge 0}.
\]
The AI has done over 3 million enumerations when $h\le 4$ but no counterexample was found. Knowing that properness may not be enough, the author push it further several times and AI found the first counter-example \ref{ex: counter-ex when not proj} where $h=6$ and number of rays is 10. This also surprised me, considering the AI missed it in the past few months. It was at this point that the author decided to write this paper.

The author do not want to claim any traditional credit for this paper, it is more like the AI brings the ideas to me than I bring the ideas to it. The author writes the whole paper including all the proofs on his own (hard to tell if the AI could write it better). The author did use the AI to help generate some graphs, especially Example \ref{ex: counter-ex when not proj}, but most of them need considerable adjustment and it was truly painful.

The purpose of this paper is more about understanding and illustrating, rather than a mark that certain question is solved. I hope people will be motivated and benefit from reading this paper.

\end{document}